\documentclass[a4paper,reqno,final]{amsart}

\usepackage{my_style}

\usepackage{soul}
\usepackage[normalem]{ulem}

\title[A fully iterative adaptive energy-based approach for elliptic PDE]{A fully iterative adaptive energy-based approach
for monotone elliptic problems}

\author{Raphael Leu \and Thomas P.~Wihler}
\address{Mathematics Institute, University of Bern, CH-3012 Switzerland}

\thanks{%
      The authors acknowledge the support of the Swiss National Science Foundation,
      Grant No. 200021\_212868.}

\subjclass[2020]{35A15, 35B38, 65J15, 65M50, 65N30.}

\keywords{Adaptive finite element methods, energy-based algorithms, iterative energy reduction schemes, energy minimization, iterative linearized Galerkin methods, conjugated gradient method, variational adaptivity, adaptive mesh refinements, monotone elliptic boundary value problems.}

\begin{document}

\begin{abstract}
    We present a fully iterative adaptive
    algorithm for the numerical minimization
    of strongly convex energy functionals
    in Hilbert spaces.
    The proposed approach, which we first present in abstract form, generates a
    hierarchical sequence of adaptively refined
    finite-dimensional approximation spaces
    and employs a (nonlinear) conjugate gradient (CG) method
    to compute suitable approximations on each space.
    A core novelty of our approach is that all components
    of the algorithm are consistently driven by
    \emph{energy reduction principles} rather than by classical a posteriori estimators.
    In particular, adaptive refinement is steered by local
    energy reduction indicators which aim to construct
    subsequent approximation spaces in a way that attains the largest potential decrease in energy.
    Likewise, the stopping criteria for the iterative solver
    are based on either relative or averaged energy reductions
    on each subspace.
    As a concrete realization, we present a concise implementation for $\spc P_1$
    finite element discretizations of second-order semilinear elliptic diffusion-reaction models, where the local indicators driving the element refinements are computed based on edge-wise energy reductions.
    Numerical experiments demonstrate that the resulting scheme
    achieves optimal convergence for various benchmark
    problems in two-dimensional polygonal domains.
\end{abstract}
\maketitle

\section{Introduction}
\label{sec:introduction}

We study the numerical minimization of energy functionals
$\E : \spc V \to \mathbb{R}$ defined on a Hilbert space $\spc V$
with inner product
$(\cdot, \cdot)_{\spc V} : \spc V \times \spc V \to \mathbb{R}$
and induced norm
$\| v \|_{\spc V} := \sqrt{(v, v)_{\spc V}}$.
Throughout, we assume the following:

\begin{assumption}
    \label{assumption}
    The energy functional $\E : \spc V \to \mathbb{R}$ is
    \begin{enumerate}[(i)]
        \item%
        \emph{Gâteaux differentiable}, with the derivative denoted by $\E':\,\spc V\to\spc V'$, $u\mapsto \E'[u]$, where $\spc V'$ signifies the dual space of $\spc V$;

        \item%
        \emph{weakly coercive}, i.e. $\E(v) \to +\infty$ as $\|v\|_{\spc V} \to \infty$;

        \item%
        \emph{strictly convex}, i.e. $\E(tu+(1-t)v)<t\,\E(u)+(1-t)\E(v)$ for any $t\in(0,1)$ and for all $u,v\in\spc V$ with $u\neq v$.
    \end{enumerate}
\end{assumption}
Under these assumptions, standard arguments of variational calculus yield the existence of a unique minimizer
$u \in \spc V$ of $\E$, i.e.
\begin{equation}\label{eq:Emin}
\E(u)<\E(v)\qquad\forall v\in\spc V,\, v\neq u,
\end{equation}
which, for G\^{a}teaux differentiable functionals, is equivalently characterized by the weak formulation
\begin{equation}
    \label{eq:full-space-problem}
    \E'[u](v) = 0
    \quad
    \forall v \in \spc V,
\end{equation}
see~\cite[Cor.~25.18]{Zeidler:IIB}. We refer to the solution of~\eqref{eq:full-space-problem}
as the \emph{full space solution}.
Clearly, all properties of Assumption~\ref{assumption}
carry over to any finite-dimensional (and thus closed) subspace
$\spc W \subset \spc V$;
in particular, the unique minimizer
$u_{\spc W} \in \spc W$
of $\E$ over $\spc W$ exists and satisfies
\begin{equation}
    \label{eq:finite-dimensional-space-problem}
    \E'[u_{\spc W}](w) = 0
    \quad
    \forall\, w \in \spc W.
\end{equation}
We call the solution of~\eqref{eq:finite-dimensional-space-problem}
the \emph{Galerkin solution in $\spc W$}.

Energy minimization problems satisfying
Assumption~\ref{assumption} naturally arise in the context of
linear and nonlinear monotone elliptic partial differential equations (PDEs) of variational type; more specifically, for the purpose of the numerical experiments in the present work (see Section~\ref{sec:numerical-experiments}), we shall focus on semilinear diffusion-reaction boundary value problems in divergence form, with monotone reaction terms, cf.~\eqref{eq:non-linear-pde}.

\subsection*{Adaptive FEM and ILG approach}
The application of adaptive finite element methods (AFEM) for the
numerical approximation of weak formulations~\eqref{eq:full-space-problem} in the context of
elliptic (and parabolic) PDE is well established, whereby they typically follow the classical
\texttt{solve}--\texttt{estimate}--\texttt{mark}--\texttt{refine} paradigm:
on a given finite element mesh, 
the Galerkin solution is computed (or an inexact approximation thereof),
and a posteriori (residual or error) estimators are employed
to flag elements for local refinement. 

In the so-called iterative linearized Galerkin (ILG) methodology,
see, e.g.~\cite{hw20-nonlinear}, the weak formulation~\eqref{eq:full-space-problem}
is solved numerically based on an instantaneous interplay of iterative solvers
(which deal with possible nonlinearities
or with linear solver errors resulting from the approximation of the
algebraic systems)
with local mesh enrichments. In combination with AFEM, this approach starts from an initial finite-dimensional subspace $\spc V_{0} \subset \spc V$, which is adaptively refined
to construct a nested sequence of subspaces
$\spc V_{0}
    \subset
    \spc V_{1}
    \subset \dots \subset
    \spc V_{N} \subset \dots \subset \spc V$.
On each space~$\spc V_{N}$, \mbox{$N \ge 0$,}
the corresponding Galerkin solution (i.e. the solution of~\eqref{eq:finite-dimensional-space-problem} on $\spc W=\spc V_N$), denoted by
$u^h_{N} \in \spc V_{N}$, is approximated by a sequence of approximations $\{u^n_{N}\}_n \subset \spc V_{N}$ that results from the application of iterative (nonlinear or linear) solvers.
In practice, the iteration is stopped after finitely many steps, yielding a final iterate $u^\star_{N}\in\spc V_N$ for each subspace.
From a theoretical point of view, this process is typically controlled by decomposing
the error into two contributions (in the simplest case), viz.
\begin{equation}
    u - u_{N}^\star
    =
    (u - u_{N}^h)
    +
    (u_{N}^h - u_{N}^\star),
\end{equation}
where the two brackets on the right-hand side contain the \emph{discretization error} and the \emph{iteration error}, respectively, with the full-space solution
$u \in \spc V$ of~\eqref{eq:full-space-problem}.
Correspondingly, in the ILG approach, the overall error is optimally bounded
by balancing both the discretization and iteration errors,
whereby the latter
depends on both the (linear or nonlinear) solvers employed 
as well as on the stopping criterion
used to select the final iterate $u^\star_{N} \in\spc V_N$.
Within an iterative adaptive context, controlling the total error therefore requires
careful monitoring of all three of these key components,
which is typically accomplished by deriving suitable a posteriori
estimates that are able to account individually for each of these error sources.
We refer to the works~\cite{hw20-nonlinear,cw17, hpw21} on the ILG approach,
or the related
works~\cite{ern2013adaptive,haberl2021convergence,JSV:10,bps25,refId0}
for further details.

\subsection*{Variational adaptivity}

Notably, the AFEM approach usually does \emph{not} explicitly exploit
a possible energy structure of the  PDE problem under consideration,
which is in contrast to the recent \emph{variational adaptivity} methodology,
see~\cite{HSW21,ahw23,hw25}, to be further developed in the current article; we also mention the paper~\cite{HMRV:24} for energy-related a posteriori bounds. Indeed, in this new approach, all algorithmic components are motivated
and guided by \emph{energy minimization principles}. Here, we may bear in mind that both the full-space solution $u \in \spc V$ of~\eqref{eq:full-space-problem} and the Galerkin solution $u^h_{N} \in \spc V_{N}$ from~\eqref{eq:finite-dimensional-space-problem} are equivalently characterized as the unique energy
minimizers over their respective spaces, cf.~\eqref{eq:Emin}. We will demonstrate in the present paper that this observation allows to address 
the above-mentioned error components using techniques inspired by (variational) principles of energy reduction, which, crucially, do not rely on classical a posteriori error estimates.
Thereby, this novel avenue proves to be a widely applicable framework for
the efficient practical approximation of linear and nonlinear PDE,
where such bounds are unavailable or difficult to obtain.

In the context of finite element discretizations,
variational adaptivity replaces the local (residual-based) error estimation step by a local energy reduction prediction.
These indicators quantify potential energy reductions
induced by using additional degrees of freedom that locally enrich the current approximation space.
This idea applies in very general terms to essentially any variational problem that
involves operators with a local structure, and to a broad class of (conforming) numerical methods:
For instance, for $hp$-type finite element methods for (self-adjoint) linear and nonlinear elliptic PDE,
the local enrichment corresponding to an increase in the
local polynomial degree or a local mesh refinement may, in both cases,
be realized in a highly effective way by adding suitable local degrees of freedom;
see~\cite{bsw25} and \cite{hhs+25} for linear elliptic PDE or for linear and
nonlinear eigenvalue problems, respectively.

\subsection*{Contribution}

Compared to the available papers on variational adaptivity~\cite{HSW21,ahw23,hw25}, where the Galerkin solution is obtained by repeatedly linearizing the nonlinear
problem and by dealing with the resulting large-scale linear systems with the aid of direct solvers, the present work advances this line of research by proposing a \emph{purely iterative energy-based method} in which
the arising discrete problems are no longer solved exactly.
Instead, we compute inexact solutions employing a (nonlinear)
conjugate gradient (CG) method in conjunction with suitable
energy-based stopping criteria.

To illustrate our approach, we present a novel realization of
variational adaptivity for $\mathbb{P}_1$
finite element discretizations in two space dimensions.
In this context, the works~\cite{HSW21,ahw23,hw25}
have used local red refinement of individual (triangular) elements in 2d to compute energy reduction indicators.
In contrast, we propose an even more local \emph{edge-based variational adaptivity strategy,}
where potential energy reductions
are associated with element interface edges, and realized
through edge bisection. This yields a more efficient and elegant algorithm:
Firstly, in contrast to element-based refinement,
each edge (or face) is refined only once during the marking step;
secondly, the resulting local linearized problems
admit simple closed-form solutions. Finally, we numerically demonstrate that variational adaptivity
retains optimal convergence rates for the sequence of final iterates,
even when no large-scale linear system is solved exactly.

\subsection*{Outline}
The remainder of this paper is organized as follows:
In Section~\ref{sec:iterative-solver-and-stopping-criteria},
we begin by considering the case where the energy $\E$ in~\eqref{eq:Emin}
is linear-quadratic;
this permits to motivate both the CG method as a natural
iterative solver for convex variational problems
and the introduction of two different energy-oriented stopping criteria for the CG scheme.
Furthermore, we highlight some key properties relevant for its interpretation as an
energy minimization method.
Section~\ref{sec:space-enrichments-and-va} introduces the concept
of variational adaptivity, which is first presented as an abstract,
energy-driven space enrichment strategy used to generate a sequence
of hierarchical subspaces based on (local) energy reduction indicators.
Subsequently, a concrete realization of this framework in the context of
finite element discretizations is developed;
in particular, we propose an edge-based variational adaptivity
strategy that may be viewed as a straightforward yet highly effective
realization of the abstract approach.
Finally, Section~\ref{sec:numerical-experiments} presents a
series of numerical experiments assessing the performance of
the overall method.
We report on convergence behavior and demonstrate the effectiveness
of variational adaptivity within a fully iterative, inexact-solver setting.

\section{Iterative Solver And Stopping Criteria}
\label{sec:iterative-solver-and-stopping-criteria}
In this Section, we introduce the iterative solver
and the associated stopping criteria for the numerical solution
of~\eqref{eq:finite-dimensional-space-problem} on a fixed
finite-dimensional subspace.
To this end, we first consider the problem in a simplified setting.
We then conclude that the strategy presented below
naturally extends to more general energy minimization problems.

\subsection{A Simplified Model Problem}
\label{sec:simplified-model-problem}
For the sake of this discussion, we temporarily assume that
$\E : \spc V \to \mathbb{R}$ is \emph{linear-quadratic},
i.e. of the form
\begin{equation}
    \label{eq:energy-linear-quadratic}
    \E(v) = \frac{1}{2} a(v,v) - \ell(v),
\end{equation}
where \(a : \spc V \times \spc V \to \mathbb{R}\) is a bounded, coercive, and symmetric
bilinear form, inducing an inner product
$(v,w)_a := a(v,w)$ on $\spc V$ with corresponding \emph{energy norm}
$\|v\|_a := \sqrt{(v,v)_a}$,
and \(\ell : \spc V \to \mathbb{R}\) denotes a bounded linear functional.
In this setting, $\E$ satisfies Assumption~\ref{assumption}.
Therefore, there exists a unique minimizer
$u \in \spc V$ of $\E$
which is equivalently characterized by the weak formulation
\begin{equation}
    \label{eq:weak-form-linear-quadratic}
    \E'[u](v)
    = a(u,v) - \ell(v)
    = 0
    \qquad \forall v \in \spc V.
\end{equation}
Again, all properties of $\E$ carry over to any
(finite-dimensional) linear subspace $\spc W \subset \spc V$
and hence, $\E$ admits a unique minimizer $u_{\spc W} \in \spc W$
over the space $\spc W$
that is equivalently characterized by the weak form
\begin{equation}
    \label{eq:finite-weak-form-linear-quadratic}
    \E'[u_{\spc W}](w)
    = a(u_{\spc W}, w) - \ell(w)
    = 0
    \qquad \forall w \in \spc W.
\end{equation}

The following lemma will be instrumental in the sequel.
\begin{lemma}
    \label{lemma:galerkin-orthogonality-general-subspace}
    Let $\spc W \subset \spc V$ be any linear subspace, and
    $u_{\spc W} \in \spc W$
    the unique minimizer of $\E$ over $\spc W$.
    Then, for all $w \in \spc W$, the identity
    \begin{equation}
        \label{eq:galerkin-orthogonality-on-general-subspace}
        \|u_{\spc W} - w\|_a^2
        =
        \|u_{\spc W}\|_a^2 + 2\E(w)
    \end{equation}
    holds true.
    Further, setting $w = u_{\spc W}$
    in~\eqref{eq:galerkin-orthogonality-on-general-subspace}
    immediately yields
    \begin{equation}
        \label{eq:negative-energy-linear-quadratic-energy}
        \E (u_{\spc W})
        =
        -\frac{1}{2}
        \|u_{\spc W}\|_a^2
        \leq 0,
    \end{equation}
    i.e. the energy of $u_{\spc W}$ is necessarily non-positive.
    
\end{lemma}

\begin{proof}
    Applying the symmetry of $a(\cdot,\cdot)$, for any $w\in\spc W$, we infer that
    \begin{align}
        \|u_{\spc W} - w\|_a^2
        &=
        a(u_{\spc W} - w, u_{\spc W} - w)
        =
        a(u_{\spc W}, u_{\spc W})
        -
        2 a(u_{\spc W}, w)
        +
        a(w, w).
    \end{align}
    Using that $u_{\spc W}$ solves
    \eqref{eq:finite-weak-form-linear-quadratic}, we obtain
        \begin{align}
        \|u_{\spc W} - w\|_a^2
        &=
        \|u_{\spc W}\|_a^2
        -
        2 \ell(w)
        +
        a(w, w)
        =
        \|u_{\spc W}\|_a^2 + 2 \E(w),
    \end{align}
    which completes the argument.
\end{proof}

\subsection{Galerkin Discretization}

From now on, let $\spc V_N \subset \spc V$ be an
$N$-dimensional linear subspace and
$\set B_N := \{\phi_1, \dots, \phi_N\} \subset \spc V_N$
a basis of $\spc V_N$.
By $u^h_N \in \spc V_N$ we denote the unique minimizer of
$\E$ over $\spc V_N$.
In this finite-dimensional setting,
$u^h_{N} \in \spc V_N$ can be written as
\begin{equation}
    \label{eq:coeff}
    u_N^h
    =
    \sum_{k=1}^N (\mat U_N^h)_k \phi_k,
\end{equation}
where
$(\mat U_N^h)_k \in \mathbb{R}$, $k=1, \dots, N$,
denote the entries of the coefficient vector
$\mat U_N^h \in \spc R^N$. Defining the stiffness matrix $\mat A \in \spc R^{N \times N}$
and the load vector $\mat b \in \spc R^N$ by
\begin{equation}
\mat A_{ij} := a(\phi_i, \phi_j)\quad\text{and}\quad
\mat b_j := \ell(\phi_j),\quad 1\le i,j\le N,
\end{equation}
respectively, the discrete weak
formulation~\eqref{eq:finite-weak-form-linear-quadratic}
with $\spc W = \spc V_N$
is equivalent to the linear system
\begin{equation}
    \label{eq:linear-algebra-problem}
    \mat A \mat U^h_N = \mat b,
\end{equation}
where $\mat A$ is symmetric positive-definite (SPD).

\begin{remark}
    \label{remark:motivation-iterative-procedure}
    The solution of
    \eqref{eq:linear-algebra-problem} may, in principle,
    be computed using a direct solver. However, for the purpose of the current work,
    we take into account the case where the space dimension $N$ becomes large enough
    to motivate the use of an iterative linear solver;
    in the context of finite element discretizations,
    this may become useful or even necessary, for instance,
    if strongly refined finite element meshes
    (in particular, in space dimensions $d\ge 2$) are involved,
    or if higher-order approximations are applied.
\end{remark}

\subsection{Iterative Solver}
\label{sec:iterative-solver}

As motivated in Remark~\ref{remark:motivation-iterative-procedure},
this subsection considers the numerical approximation of
$u^h_N \in \spc V_N$ by
generating an energy minimizing sequence
$\{u^n_{N}\}_n \subset \spc V_{N}$. 
Recalling the linear algebra formulation~\eqref{eq:linear-algebra-problem}
with $\mat A$ SPD,
it is natural to employ the CG algorithm.
Then, the $n$-th element of this sequence may be written as
\begin{equation}
    u^n_{N}
    =
    u^{0}_{N} + \sum_{\ell = 1}^n p_{N}^\ell,
\end{equation}
where $u^0_{N} \in \spc V_{N}$ is an initial guess,
and the vectors $p_{N}^\ell := u^\ell_{N} - u^{\ell-1}_{N}$
represent the updates in each iteration step, which are
pairwise orthogonal with respect to the inner product $(\cdot, \cdot)_a$.
Note that the CG method converges to the
Galerkin solution $u^h_{N}$
of~\eqref{eq:finite-weak-form-linear-quadratic}
in at most $N$ steps (in exact arithmetic),
whence we observe the identity
\begin{equation}
    u^h_{N} - u^n_{N}
    =
    u^{N}_{N} - u^n_{N}
    =
    \sum_{\ell = n+1}^{N} p_{N}^\ell.
\end{equation}
Further, exploiting $a$-orthogonality, we obtain
\begin{equation}
    \label{eq:hs-estimate}
    \|u^h_{N} - u^n_{N}\|_a^2
    =
    \sum_{\ell = n+1}^{N}\|p_{N}^\ell\|_a^2
    \geq
    \sum_{\ell = n+1}^{n+d}\|p_{N}^\ell\|_a^2
    =: \xi^d_{n},
\end{equation}
where the so-called \emph{Hestenes-Stiefel (HS) estimate~$\xi_{n}^d$}
represents a lower bound on the iteration error
in the energy norm (see~\cite{hs52}),
and $d\in [1:N-n]$ is the so-called \textit{delay parameter}.
The CG method can naturally be interpreted as
an energy minimization algorithm:
indeed,  in the $\ell$-th iteration,
the energy is minimized over the affine linear set
\begin{equation}
    \spc V_{N}^\ell
    :=
    \left\{
        u^0_{N} + v
    \,\middle|\,
        v \in
        \operatorname{span} \left\{ p^1_{N}, \dots, p^\ell_{N} \right\}
    \right\}.
\end{equation}
In the following lemma, we demonstrate that the
Hestenes-Stiefel estimate can likewise be interpreted
from an energy perspective.

\begin{lemma}
    \label{lemma:hestenes-stiefel-estimate-in-energy-fashion}
    The HS estimate~\eqref{eq:hs-estimate} can be written equivalently as
    \begin{equation}
        \label{eq:hs-estimate-energy}
        \xi_{n}^d
        = 2
        \left(
            \E(u_{N}^n) - \E(u_{N}^{n+d})
        \right).
    \end{equation}
\end{lemma}

\begin{proof}
    Employing Lemma~\ref{lemma:galerkin-orthogonality-general-subspace}, and making use of \eqref{eq:hs-estimate}, we observe that
    \begin{align}
        2\left(
            \E(u^n_{N}) - \E (u^{n+d}_{N})
        \right)
        &=
        \|u^h_{N} - u^n_{N}\|_a^2
        -
        \|u^h_{N} - u^{n+d}_{N}\|_a^2
        \\
        &=
        \sum_{\ell=n+1}^{N} \|p_{N}^\ell \|_a^2
        -
        \sum_{\ell=n+d+1}^{N} \|p_{N}^\ell \|_a^2
        \\
        &=
        \sum_{\ell=n+1}^{n+d} \|p_{N}^\ell \|_a^2
        =
        \xi_n^d,
    \end{align}
    which yields the claim.
\end{proof}

Summarizing, in order to minimize the energy over $\spc V_N$,
we employ the CG algorithm as iterative solver which,
in the $\ell$-th step, minimizes the energy over the affine linear set
$\spc V^{\ell}_{N}$ and whose iteration error can be estimated
from below by the HS estimate~\eqref{eq:hs-estimate-energy}.

\subsection{Stopping Criteria}
\label{sec:stopping-criteria}

We will now present two stopping criteria
that allow to obtain an appropriate final iterate
$u^\star_{N}$ in the sequence
$\{u^n_{N}\}_n \subset \spc V_{N}$ generated by the CG algorithm.

\subsection*{(a) Energy Tail-Off Stopping Criterion}
The first stopping criterion we
introduce has been already employed in the work~\cite{ahw23}.
Roughly speaking, the idea is to stop the iteration as soon as the sequence of energy values $\{\E(u^n_{N})\}_n \subset \spc R$ begins to level-off.
To formalize this approach, let $\Phi^n_{N}$ denote the
accumulated energy decay up to the $n$-th iteration, i.e.
for $n \geq 1$, let
\begin{equation}
    \Phi^n_{N} := \E (u^0_{N}) - \E(u^n_{N})\ge 0.
\end{equation}
The iteration is stopped as soon
as  it holds
\begin{equation}
    \label{eq:original-energy-flattening-off-criterion}
    \E(u^{n-1}_{N}) - \E(u^n_{N})
    <
    \alpha_{\E} \Phi^n_{N},
\end{equation}
where $\alpha_{\E} \in (0, 1]$ is a prescribed control parameter.
In this work, however,
we slightly modify this criterion by dividing the
right-hand side of
\eqref{eq:original-energy-flattening-off-criterion}
by the iteration number $n$,
thereby turning the right-hand side into an averaged quantity
over all iterations that is more restrictive;
specifically, our new stopping criterion reads
\begin{equation}
    \label{eq:energy-tail-off-stopping-criterion}
    \E(u^{n-1}_{N}) - \E(u^n_{N})
    <
    \alpha_{\E} \frac{\Phi^n_{N}}{n},
\end{equation}
with $\alpha_{\E} \in (0, 1]$ a
control parameter as above.

\subsection*{(b) Relative Energy Reduction Stopping Criterion}
The second stopping criterion we introduce
is inspired by an approach proposed in \cite{ari04}.
For any $u^n_{N} \in \spc V_{N}$
(not necessarily resulting from the CG algorithm),
they show that, whenever we have the bound
\begin{equation}
    \label{eq:arioli-assumption}
    \|u^h_{N} - u^n_{N}\|_a \leq \gamma \|u^h_{N}\|_a,
\end{equation}
for some parameter $\gamma>0$, then it follows that
\begin{equation}
    \label{eq:arioli-consequence}
    \|u - u^n_{N}\|_a \leq \gamma \|u\|_a + 2\|u-u^h_{N}\|_a,
\end{equation}
where $u\in\spc V$ and $u^h_{N}\in\spc V_{N}$ are
the full space solution and Galerkin solution,
respectively,
see~\cite[Eq.~(11)]{ari04}.
Hence,
condition~\eqref{eq:arioli-assumption}
may be regarded as an academic stopping criterion which,
when satisfied, guarantees the inequality~\eqref{eq:arioli-consequence}.
Let us first discuss why this implication is of interest,
before addressing the task of transforming~\eqref{eq:arioli-assumption}
into a practical stopping criterion.
To this end, suppose for a moment that a strategy for generating
a sequence of finite-dimensional subspaces
$\{\spc V_N\}_{N \geq 0} \subset \spc V$ is available
such that the associated Galerkin solutions $u^h_N \in \spc V_N$
satisfy the asymptotic estimate
\begin{equation}
    \|u - u^h_N\|_a
    =
    \mathcal{O}\!\left( \bigl(\dim \spc V_N\bigr)^{-r} \right),
\end{equation}
for a rate $r\ge 1$.
The construction of such a sequence is
discussed in Section~\ref{sec:space-enrichments-and-va}.
Then, on each subspace~$\spc V_{N}$,
choosing $\gamma \propto \left(\text{dim} \spc V_N\right)^{-r}$,
and iterating until~\eqref{eq:arioli-assumption} is satisfied
for the final approximation $u^\star_{N}$,
the implication~\eqref{eq:arioli-consequence} yields
\begin{equation}
    \|u - u^\star_{N}\|_a
    =
    \mathcal{O}\!\left( \bigl(\dim \spc V_N\bigr)^{-r} \right),
\end{equation}
thus recovering the optimal convergence rate.
With this motivation in place, we now turn to the
task of transforming~\eqref{eq:arioli-assumption}
into a practical stopping criterion by deriving
computable estimates for both sides of the inequality.

A computable lower bound for the right-hand side of~\eqref{eq:arioli-assumption}
can be obtained from Lemma~\ref{lemma:galerkin-orthogonality-general-subspace}. Indeed, from
\begin{equation}
    \|u_{N}^h - u_{N}^n\|_a^2 - 2 \E(u_{N}^n)
    =
    \|u_{N}^h\|_a^2,
\end{equation}
we immediately deduce the estimate
\begin{equation}
    \label{eq:lower-bound-a-norm-galerkin}
    - 2 \E(u_{N}^n) \leq \|u_{N}^h\|_a^2.
\end{equation}
This inequality provides a practical and computable lower
bound for the right-hand-side of \eqref{eq:arioli-assumption},
i.e. the energy norm of the Galerkin solution.

Deriving an a posteriori upper bound for the
left-hand side of~\eqref{eq:arioli-assumption},
i.e. the iteration error, is a far more challenging task.
In fact, even when restricting ourselves to iterates
generated by the CG algorithm,
existing techniques typically depend on a priori
spectral information of the system matrix, which entails non-negligible computational cost (see, e.g.~\cite{MT24}), and,
more critically for our work, disregard the underlying energy structure of the problem.
For these reasons, following the approach proposed in~\cite{ari04},
we rely on the HS estimate
as an a posteriori estimator for the iteration error;
although the HS estimate yields only a lower bound on the true iteration error,
cf.~\eqref{eq:hs-estimate}
(whilst an upper bound would be required for reliability),
this approach, which was also used in the original work~\cite{ari04},
has proven effective in our numerical experiments also.
Further, in conjunction with
Lemma~\ref{lemma:hestenes-stiefel-estimate-in-energy-fashion},
employing the HS estimate fits naturally into our energy-based framework.

Summarizing, we substitute the right-hand side of~\eqref{eq:arioli-assumption}  
with its lower bound~\eqref{eq:lower-bound-a-norm-galerkin},  
while the left-hand side of~\eqref{eq:arioli-assumption} is approximated by
the HS estimate in its energy form
(cf.~Lemma~\ref{lemma:hestenes-stiefel-estimate-in-energy-fashion}).
Further, as motivated above, for the $\mathbb{P}_1$ finite element discretizations in the numerical experiments in Section~\ref{sec:numerical-experiments}, we set the parameter $\gamma$ in~\eqref{eq:arioli-assumption}
to $\gamma^2 := \nicefrac{\alpha_{\gamma}}{\text{dim}\spc V_N}$ as we expect an optimal approximation rate of $r=\nicefrac12$;
here, $\alpha_\gamma \in (0, 1]$ is a control parameter.
Finally, this yields the following computable a posteriori stopping criterion:
\begin{equation}
    \label{eq:energy-arioli-stopping-criterion}
    \E (u_{N}^n) - \E (u_{N}^{n+d})
    \leq
    - \frac{\alpha_\gamma}{\text{dim} \spc V_N} \E(u_{N}^n),
\end{equation}

\noindent
where $d$ denotes the delay parameter used in the HS estimate.
This means that
the iteration is terminated at the $n$-th step once the relative
energy reduction---measured between the current iterate $u_{N}^n$
and the delayed iterate $u_{N}^{n+d}$---falls below the threshold
$\gamma^2 = \nicefrac{\alpha_\gamma}{\text{dim}\spc V_N}$.
We emphasize that this stopping criterion is mathematically equivalent
to the one proposed in~\cite{ari04}.
The key distinction is that we have reformulated it to
emphasize its energy-based interpretation.

\begin{remark} 
\label{remark:stopping-criteria}
We conclude this Section with a few practical aspects.
\begin{enumerate}[(a)]
\item In the article~\cite{ari04} a strategy to 
adaptively choose the delay parameter
on an ad hoc basis has been presented: Starting from
an initial value $d_{\text{init}} \geq 1$,
the delay parameter is updated by adding a
fixed increment $\Delta d \geq 1$ whenever the condition
\begin{equation}
    \label{eq:adaptive-delay}
    \xi^d_{n+1} > \tau \xi^d_n
\end{equation}
is met, where $\tau > 1$ is a control parameter, i.e., we set $d \leftarrow d + \Delta d$ as long as~\eqref{eq:adaptive-delay} holds true.
The convergence test~\eqref{eq:energy-arioli-stopping-criterion}  
is only applied once~\eqref{eq:adaptive-delay} is no longer satisfied.
In this paper, we adopt this adaptive strategy
in our numerical experiments in Section~\ref{sec:numerical-experiments}.

\item We have also tested a combined
stopping criterion that consecutively applied the energy tail-off
criterion~\eqref{eq:energy-tail-off-stopping-criterion}
for global convergence and the relative energy
reduction criterion~\eqref{eq:energy-arioli-stopping-criterion}
for local convergence (or the bound~\eqref{eq:energy-arioli-stopping-criterion-abs} for more general situations to be discussed below in Section~\ref{sec:extended-applicability}). The resulting performance and convergence behavior were comparable to those
obtained with the energy tail-off criterion alone,
thereby providing no substantial improvements with regard to
the qualitative or quantitative regime of the scheme.
We therefore omit these results from the present discussion.

\item%
We note that the use of a preconditioner in the CG algorithm
does not interfere with the stopping criteria
presented in this Section.
In particular, when employing the preconditioned conjugate gradient (PCG)
method instead of CG, the stopping criteria can be applied exactly as
described above without any modifications.
Indeed, if $\mat{x}_k$, $\hat{\mat{x}}_k$ denote the
iterates of the CG and corresponding PCG algorithm,
it can be verified that it holds
\begin{equation}
    \E (\mat{x}_k)
    =
    \hat{\E} (\hat{\mat{x}}_k),
\end{equation}
where $\E$ and $\hat \E$ denotes the energy of the system without and with preconditioning,
respectively.

\item%
To reduce the sensitivity of the stopping
criterion to local energy flattenings, one may evaluate
criterion~\eqref{eq:energy-tail-off-stopping-criterion} in batches,
that is, only for every $n_{\text{batch}}$-th iterate, where
$n_{\text{batch}} \geq 1$ is an integer-valued control parameter.
This strategy has been adopted in
Section~\ref{sec:numerical-experiments}.

\item%
In order to evaluate stopping
criterion~\eqref{eq:energy-arioli-stopping-criterion}
for the $n$-th iterate,
we need to compute the $(n+d)$-th iterate as well.
As we focus on assessing the effectiveness of the stopping criterion
from an academic perspective,
the errors computed and shown in Section~\ref{sec:numerical-experiments}
always correspond to the $n$-th iterate,
even though the $(n+d)$-th iterate is available.
In a practical setting, however,
one would naturally go along with the $(n+d)$-th iterate.
Finally, we remark that the iteration counts shown in the corresponding
plots do include the additional $d$ steps required to evaluate
the relative energy reduction
stopping criterion~\eqref{eq:energy-arioli-stopping-criterion}.

\end{enumerate}
\end{remark}

\subsection{Application to more general minimization problems}
\label{sec:extended-applicability}

The iterative approximation of the Galerkin solution
of the linear-quadratic energy functional~\eqref{eq:weak-form-linear-quadratic} presented above,
including the choice of iterative solver and stopping criteria,
is primarily motivated by energy minimization principles.
This observation naturally suggests extending this iterative strategy
to more general energy minimization problems,
where the underlying energy functional
merely satisfies Assumption~\ref{assumption}.
In this setting, however, a few adjustments need to be made:
\begin{enumerate}[(a)]

\item  Firstly, the a posteriori stopping criterion~\eqref{eq:energy-arioli-stopping-criterion} requires a slight modification. Indeed, for a monotonically decreasing sequence
    $\{\E(u_N^n)\}_n$, note that the left-hand side of
    \eqref{eq:energy-arioli-stopping-criterion} is non-negative (cf.~Lemma~\ref{lemma:galerkin-orthogonality-general-subspace}), and thus renders~\eqref{eq:energy-arioli-stopping-criterion}
    ineffective whenever the energy minimum is positive (which, in contrast to the linear--quadratic case, is possible). A simple yet effective remedy is to employ absolute values, that is,
    to replace~\eqref{eq:energy-arioli-stopping-criterion} by
    \begin{equation}
        \label{eq:energy-arioli-stopping-criterion-abs}
        \left|
            \E(u_N^n) - \E(u_N^{n+d})
        \right|
        \le
        \frac{\alpha_\gamma}{\text{dim} \spc V_N}
        \left|
            \E(u_N^n)
        \right|.
    \end{equation}
    This modification allows the stopping criterion to be interpreted
    in terms of the relative change in energy over $d$ iterations and
    extends its applicability to more general energy minimization problems.
    We emphasize that, in the linear-quadratic setting and sufficiently close
    to the minimizer, criterion~\eqref{eq:energy-arioli-stopping-criterion}
    is \emph{equivalent} to~\eqref{eq:energy-arioli-stopping-criterion-abs}.

\item Secondly, the linear CG algorithm is no longer applicable when the weak formulation~\eqref{eq:finite-dimensional-space-problem} is nonlinear.
Instead, suitable alternatives are provided by
nonlinear CG methods
(for a comprehensive overview of nonlinear CG algorithms
and their global convergence properties, we refer to~\cite{HZ06}).
In corresponding numerical experiments
in Section~\ref{sec:nonlinear-experiments},
we employ the existing implementation
\verb|scipy.optimize.fmin_cg|~\cite{scipy},
which provides a nonlinear CG method based on the approach
of Polak and Ribi\`ere~\cite{nocedal}.
However, the energy-based stopping criteria introduced above do,
in principle, remain applicable without any modifications
(cf.~Section~\ref{sec:nonlinear-experiments}).
For future reference we also introduce
the \emph{default} stopping criterion implemented in
\verb|scipy.optimize.fmin_cg|, which reads
\begin{equation}
    \label{eq:stopping-criterion-default}
    \|\nabla \E(u^n_{N})\|_{\infty} \le g_{\text{tol}},
\end{equation}
where
$\nabla \E(u^n_{N}) \in \mathbb{R}^{\text{dim}\spc V_N}$
denotes the gradient of $\E$
evaluated at the current iterate $u^n_{N}$,
and $g_{\text{tol}} \in [0, \infty)$
is a user-chosen control parameter.
The default stopping criterion terminates
the iteration at the $n$-th step, as soon
as~\eqref{eq:stopping-criterion-default} is satisfied.
\end{enumerate}

\section{Space enrichment and variational adaptivity}
\label{sec:space-enrichments-and-va}

In this Section, given a (finite-dimensional) subspace
$\spc V_N \subset \spc V$ and a final iterate $u^\star_N \in \spc V_N$,
we present an energy-based, adaptive enrichment strategy to
construct the subsequent (refined) subspace
$\spc V_{N+1} \subset \spc V_N$ aiming for the largest possible local energy reduction. 
We first present the algorithm in abstract form
for a general Hilbert space $\spc V$,  
then give a concrete realization for $\mathbb{P}_1$
finite element spaces on two-dimensional domains,  
and finally discuss computational aspects of the proposed framework.
Throughout this Section, we assume that
$\E : \spc V \to \mathbb{R}$ satisfies Assumption~\ref{assumption}.

\subsection{Variational Adaptivity}
\label{sec:variational-adaptivity}

Given a subspace $\spc V_{N} \subset \spc V$
and the final iterate $u^\star_N \in \spc V_N$,
let $\mat \xi = \{ \xi_1, \dots, \xi_L \} \subset \spc V$ be 
a set of linearly independent elements such that
the intersection $\spc V_N \cap \mat \xi = \emptyset$
is empty, and define the (low-dimensional) subspace
$\spc Y := \text{span} \{ u^\star_{N}, \xi_1, \dots, \xi_L \}$.
In this setting, we call $\mat \xi$ an \textit{enrichment set}
and $\spc Y$ a \textit{(local) enrichment space};
in the context of finite element discretizations,
the space $\spc Y$ is spanned by one global mode,
namely the approximation~$u_{N}^\star$,
and $L$ additional degrees of freedom that act locally in the
vicinity of a specific element in order to detect and resolve
potential fine scale features in the full-space solution.
By considering~\eqref{eq:finite-dimensional-space-problem}
with $\spc W = \spc Y$, we obtain the (low-dimensional) problem of finding $u^h_{\spc Y} \in \spc Y$ such that it holds
\begin{equation}
    \label{eq:local-general-nonlinear-weak-problem}
    \E'\left[u^h_{\spc Y}\right](y)
    = 0
    \quad
    \forall y \in \spc Y.
\end{equation}
We call $u^h_{\spc Y}$ a \emph{locally improved solution},
and note that it is equivalently characterized as the
unique minimizer of $\E$ over $\spc Y$; in particular,
since $u^\star_{N} \in \spc Y$,
we immediately obtain the energy reduction property
\begin{equation}
    \E(u^h_{\spc Y}) \leq \E(u^\star_{N}).
\end{equation}
We can thus define the
non-negative \emph{energy decay}
\begin{equation}
    \label{eq:energy-decay}
    \Delta \E (u^\star_{N}, \mat \xi)
    :=
    \E (u^\star_{N}) - \E (u^h_{\spc Y})\ge 0.
\end{equation}
The energy decay
$\Delta \E (u^\star_{N}, \mat \xi)$
provides a computable a priori lower bound to the
\emph{potential energy reduction} resulting from
enriching the space $\spc V_{N}$ with the
enrichment set $\mat \xi$, i.e. by switching from
$\spc V_{N}$ to the augmented space
$\spc V_{N} \oplus \text{span}~\mat \xi$.

The above considerations motivate the following strategy:
Let $\{\mat \xi_\ell\}_{\ell=1, \dots, M}$ be
a family of enrichment sets.
For each of the enrichment sets
we can compute the corresponding energy decay~\eqref{eq:energy-decay}.
Then, choose an index subset $\set I \subset \{1, \dots, M\}$
of minimal cardinality that fulfills the classical
Dörfler marking criterion~\cite{dör96}
; in the language of energy decays, this criterion reads
\begin{equation}
    \label{eq:dörfler}
    \sum_{\ell \in \mathcal{I}} \Delta \E (u^\star_{N}, \mat \xi_\ell)
    \geq 
    \theta \sum_{\ell = 1}^{M}
    \Delta \E (u^\star_{N}, \mat \xi_\ell),
\end{equation}
where $\theta \in (0, 1)$ is the \textit{Dörfler parameter}.
Finally, define
\begin{align}
    \label{eq:space-enrichment}
    \spc V_{N+1}
    & :=
    \spc V_{N} 
    \oplus \text{span}~
    \bigcup_{\ell \in \mathcal{I}}
    \mat \xi_\ell,
    \intertext{and apply the natural embedding}
    \label{eq:canonical-embedding}
     u^0_{N+1}
    & :=
    u^\star_{N}\hookrightarrow \spc V_{N+1},\nonumber
\end{align}
to define a starting point for approximating the minimizer
$u^h_{N+1}$ of $\E$ within the subsequent subspace $\spc V_{N+1}$.
The resulting process that iteratively generates a nested sequence
of subspaces is outlined in Algorithm~\ref{alg:variational-adaptivity}
and is referred to as \textit{variational adaptivity}.
The application of this algorithm
within a finite element context is discussed in
Section~\ref{sec:eva}.

\begin{remark}\label{rem:solve-local-system}
In the linear-quadratic case (cf. Section~\ref{sec:simplified-model-problem}),  
problem~\eqref{eq:local-general-nonlinear-weak-problem}
reduces to an $(L_j+1)\times(L_j+1)$ linear system,
where $L_j$ is the number of elements in $\mat\xi_j$.
Solving this system exactly is computationally inexpensive,
as $L_j$ is usually small. More generally, however, if the underlying energy only
satisfies Assumption~\ref{assumption},  
it is typically impossible to solve the local
problem~\eqref{eq:local-general-nonlinear-weak-problem} in closed form as it may be nonlinear.
In such cases, one can accurately approximate the exact solution by applying an appropriate nonlinear solver, 
e.g., a Newton-type scheme, that results in a sequence of $(L_j+1)\times(L_j+1)$ linear systems.
\end{remark}

\begin{algorithm}
    \caption{Variational adaptivity}
    \label{alg:variational-adaptivity}
    \begin{algorithmic}[1]  
        \STATE%
        Prescribe the Dörfler parameter $\theta \in (0, 1)$.
        \STATE%
        Input a finite-dimensional subspace $\spc V_{N} \subset \spc V$,
        a suitable approximation $u^\star_{N} \in \spc V_{N}$,
        and a family of enrichment sets
        $\{ \mat \xi_\ell \}_{\ell=1,\dots, M}$
        not contained in $\spc V_{N}$.
        \FOR{$\ell=1, \dots, M$}
            \STATE%
            \label{alg:variational-adaptivity-solve}
            SOLVE the (low-dimensional)
            problem~\eqref{eq:local-general-nonlinear-weak-problem}
            for $u_{\spc Y}^h \in \spc Y$.
            \STATE%
            \label{alg:variational-adaptivity-compute}
            Compute the corresponding energy decay
            \eqref{eq:energy-decay}.
        \ENDFOR
        \STATE%
        CHOOSE an index subset $\set I \subset \{1, \dots, M\}$
        based on the Dörfler marking criterion~\eqref{eq:dörfler}.
        \STATE%
        ENRICH the space $\spc V_{N}$ according to
        \eqref{eq:space-enrichment}, thereby yielding the subsequent
        linear subspace $\spc V_{N+1} \supset \spc V_{N}$.
    \end{algorithmic}
\end{algorithm}

\subsection{Edge-based Variational Adaptivity}
\label{sec:eva}

Let $\Omega \subset \mathbb{R}^2$ be an open and bounded polygonal
domain and $\Hone$ the first-order Sobolev
space with zero trace on $\partial \Omega$.
Then, in this subsection, we illustrate how the abstract framework
introduced in Section~\ref{sec:variational-adaptivity}
can be used to adaptively refine a two-dimensional simplicial mesh
underlying a $\mathbb{P}_1$ finite element discretization of $\Hone$.
Specifically, we discuss how edge bisections
can be interpreted as local space enrichments
and the resulting augmentation~\eqref{eq:space-enrichment}
as mesh refinement.
In combination with the iterative solver developed in
Section~\ref{sec:iterative-solver-and-stopping-criteria}, this
enables the construction of
an adaptive and fully iterative
algorithm
for minimizing energy functionals over $\Hone$ using
$\spc P_1$ finite element discretizations
in two spatial dimensions,
to be employed in numerical experiments in
Section~\ref{sec:numerical-experiments}.
Note that this implementation of variational adaptivity
contrasts earlier works~\cite{HSW21,ahw23,hw25},
which have employed a similar approach,
especially in the sense that
they have used \emph{local red refinement} of individual
elements to compute the energy reduction indicators,
whereas we rely on more localized edge bisections.

\subsubsection{The $\mathbb{P}_1$ Finite Element Space}
For $N \geq 0$, we consider a triangulation $\mathcal{T}_N$ consisting of disjoint
open triangles $\{T\}_{T\in\mathcal{T}_N}$ that form a regular and shape-regular
mesh partition of the domain $\Omega$.
The corresponding $\spc P_1$ conforming finite element space is then given by
\begin{equation}
    \label{eq:p1-fem-space}
    \spc P_1(\mathcal{T}_N)
    :=
    \left\{
            \varphi \in \Hone
            \, \Big| \,
            \forall
            T \in \mathcal{T}_N :
            \varphi|_T \in \mathbb{P}_1(T)
    \right\}.
\end{equation}
Further, let $\mathcal{E}_N$ denote the set of all edges in the triangulation~$\mathcal{T}_N$,
and $\mathcal{V}_N$ the set of all of its vertices.
For future reference, we define
$\set E^I_N \subset \set E_N$ and $\set V^I_N \subset \set V_N$
to be the set of all interior edges and of all interior vertices, respectively.
In this context, we define $d_N := \text{dim} \spc V_N$
to be the number of degrees of freedom in
the discrete space $\spc P_1(\mathcal{T}_N)$, i.e.
the number of interior vertices
$d_N = \# \set V^I_N$.
As a basis of $\spc P_1(\mathcal{T}_N)$ we introduce the usual
Lagrange hat basis functions
$\{ \phi_i \}_{i=1, \dots, d_N} \subset \spc P_1(\mathcal{T}_N)$,
uniquely defined by $\phi_i(z_j) = \delta_{ij}$,
$z_j \in \set V^I_N$, where $\delta_{ij}$ denotes the Kronecker delta.
Then, any discrete approximation $u_N \in \spc P_1(\mathcal{T}_N)$
can be written as
\begin{equation}
    \label{eq:discrete-element-p1-fem}
    u_N = \sum_{j=1}^{d_N} U^j_N \phi_j,
\end{equation}
where $U^j_N \in \mathbb{R}$, $1\leq j \leq d_N$, cf.~\eqref{eq:coeff}.

\begin{remark}
    \label{remark:va-not-restricted-to-p1}
    Throughout this work, we restrict ourselves to piecewise linear finite elements.  
    This choice allows us to present the main concepts of our approach in the
    least technical setting.
    The variational adaptivity framework, however, is not limited to linear elements.
    In fact, the ideas can be generalized to higher-order
    finite elements and even to $hp$-methods (see, e.g.~\cite{bsw25}).
\end{remark}

\subsubsection{Edge-Based Space Enrichments}
\label{sec:eva-space-enrichments}

The following procedure is a concrete, edge-based
realization of \textit{variational adaptivity},
translating edge bisections into local space enrichments.
For each interior edge $E \in \set E^I_N$,
we define a corresponding enrichment set~$\mat \xi_E$
(here, a singleton)
that is constructed via bisection of the edge $E$ as follows:
Let $E = \operatorname{conv}\{z_i, z_j\} \in \set E^I_N$
be any interior edge connecting two vertices $z_i, z_j \in  \set V_N$,
where $\operatorname{conv}\{\cdot\}$ denotes (the open interior of) the convex hull.
Then, consider the \textit{local patch} $\omega_E$ comprising the two elements
$T_{\sharp}, T_{\flat} \in \set T_N$ that share the edge $E$, i.e.
$\omega_E := \{T_{\sharp}, T_{\flat}\}$.
Further, let $z_{\sharp}, z_{\flat} \in \set V_N$ denote the remaining vertices
of the local patch on either side of $E$;
a sketch of this configuration is shown in Figure~\ref{fig:eva}.
\begin{figure}
    \centering
    \begin{minipage}{.5\textwidth}
        \centering
        \begin{tikzpicture}
    \coordinate (C) at (0,0);
    \coordinate (L) at (-2.5,0);
    \coordinate (R) at (2.5,0);
    \coordinate (U) at (0,2.5);
    \coordinate (D) at (0,-2.5);

    \node at (L) [circle, fill=black, inner sep=1pt] {};
    \node at (R) [circle, fill=black, inner sep=1pt] {};
    \node at (D) [circle, fill=black, inner sep=1pt] {};
    \node at (U) [circle, fill=black, inner sep=1pt] {};

    \draw[thick, fill=blue, fill opacity=0.2] (L) -- (D) -- (U) -- cycle;
    \draw[thick, fill=blue, fill opacity=0.2] (R) -- (D) -- (U) -- cycle;

    \node[above right] at (U) {$z_j$};
    \node[below right] at (D) {$z_i$};
    \node[right] at (R) {$z_{\sharp}$};
    \node[left] at (L) {$z_{\flat}$};

    \node at (barycentric cs:L=1,D=1,U=1) {$T_{\flat}$};
    \node at (barycentric cs:R=1,D=1,U=1) {$T_{\sharp}$};
\end{tikzpicture}
    \end{minipage}%
    \begin{minipage}{0.5\textwidth}
        \centering
        \begin{tikzpicture}
    \coordinate (C) at (0,0);
    \coordinate (L) at (-2.5,0);
    \coordinate (R) at (2.5,0);
    \coordinate (U) at (0,2.5);
    \coordinate (D) at (0,-2.5);

    \draw[thick, fill=blue, fill opacity=0.2] (L) -- (U) -- (D) -- cycle;
    \draw[thick, fill=blue, fill opacity=0.2] (R) -- (U) -- (D) -- cycle;

    \draw[dotted, thick] (L) -- (C);
    \draw[dotted, thick] (R) -- (C);

    \node at (L) [circle, fill=black, inner sep=1pt] {};
    \node at (R) [circle, fill=black, inner sep=1pt] {};
    \node at (D) [circle, fill=black, inner sep=1pt] {};
    \node at (U) [circle, fill=black, inner sep=1pt] {};
    
    \filldraw[fill=blue!20, draw=black, line width=1pt] (C) circle[radius=2pt];

    \node[above right] at (U) {$z_j$};
    \node[below right] at (D) {$z_i$};
    \node[right] at (R) {$z_{\sharp}$};
    \node[left] at (L) {$z_{\flat}$};
    \node[above right] at (C) {$\bar z_{ij}$};

    \node at (barycentric cs:L=1,C=1,U=1) {$T_{\flat}^2$};
    \node at (barycentric cs:L=1,C=1,D=1) {$T_{\flat}^1$};
    \node at (barycentric cs:R=1,C=1,U=1) {$T_{\sharp}^2$};
    \node at (barycentric cs:R=1,C=1,D=1) {$T_{\sharp}^1$};
\end{tikzpicture}
    \end{minipage}
    \caption{
        Local patches associated to an interior edge
        $E = (z_i, z_j)$.
        Left: local patch $\omega_E$ before the local refinement.
        Right: local modified patch $\widetilde{\omega}_E$
        obtained by bisection of the edge $E$ and removing the
        hanging node $\bar z_{ij}$ by connecting it to both vertices $z_{\sharp}$ and $z_{\flat}$.
    }
    \label{fig:eva}
\end{figure}
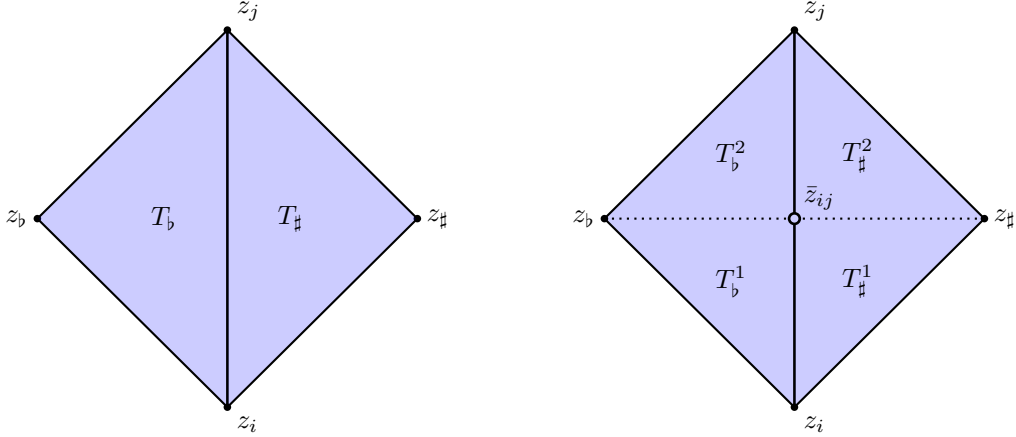
Define $\bar z_{ij} \in E$ as the midpoint of $E$, i.e.
$\bar z_{ij} := \nicefrac{1}{2} (z_i + z_j)$,
and consider the open triangles
\begin{align}
T_{\sharp}^1 &:= \text{conv}~\{z_i, \bar z_{ij}, z_{\sharp}\},&
T_{\sharp}^2 := \text{conv}~\{z_j, \bar z_{ij}, z_{\sharp}\},\\
T_{\flat}^1 &:= \text{conv}~\{z_i, \bar z_{ij}, z_{\flat}\},&
T_{\flat}^2 := \text{conv}~\{z_j, \bar z_{ij}, z_{\flat}\},
\end{align}
which are the triangles obtained by splitting both $T_{\sharp}$ and $T_{\flat}$
through connections from $\bar z_{ij}$ to $z_{\sharp}$ and $z_{\flat}$, respectively.
We define the modified (refined) patch as
$\widetilde{\omega}_E := \{T_{\sharp}^1, T_{\sharp}^2, T_{\flat}^1, T_{\flat}^2\}$.
Next, we introduce a hat function
$\phi_E \in \Hone$ associated to the new vertex $\bar z_{ij}$ that is
uniquely defined by requiring $\phi_E|_\tau \in \spc P_1(\tau)$
for all elements $\tau \in \widetilde{\omega}_E \cup (\set T_N \setminus \omega_E)$,
and
\begin{equation}
    \forall z \in \set V_N\cup\{\bar z_{ij}\} : \quad\phi_E(z)
    = 
    \begin{cases}
        1&\text{if } z=\bar z_{ij},\\
        0&\text{otherwise.}
    \end{cases}
\end{equation}
We then define the enrichment set associated with
edge $E$ as the singleton $\mat \xi_E := \{\phi_E\}$.
Following this procedure for all interior edges yields a
family of enrichment sets $\{\mat \xi_E\}_{E \in \set E^I}$.
Then, we proceed according to Algorithm~\ref{alg:variational-adaptivity},
i.e. we compute the non-negative energy decays
$\Delta \E (u^\star_{N}, \mat \xi_E)$ from~\eqref{eq:energy-decay}
and select a suitable subset that fulfills the Dörfler
marking criterion~\eqref{eq:dörfler}.
The corresponding edges we refer to as \textit{marked edges}.
The proposed edge-based variational adaptivity procedure is summarized in
Algorithm~\ref{alg:edge-based-variational-adaptivity}.

\subsection{Computational Aspects}

The practical application of the abstract framework above,
in particular regarding the solution
of~\eqref{eq:local-general-nonlinear-weak-problem}
and the space enrichment~\eqref{eq:space-enrichment},
will be discussed in the specific context of $\mathbb{P}_1$
finite element discretizations in the following subsections.

\subsubsection{Mesh Refinement}

When working with $\spc P_1$ finite element spaces,
the abstract space enrichment procedure in~\eqref{eq:space-enrichment}
needs some fine adjustments.
To see why, recall that, for any $h$-FEM,
additional degrees of freedom are introduced through mesh refinement.
Such refinements, however, must be performed carefully
in order to maintain mesh regularity and to prevent degeneration.
Edge-based Variational Adaptivity, as outlined in Section~\ref{sec:eva},
leaves us with a subset of marked edges.
We make sure to both bisect all marked edges and
to preserve the mesh quality by inputting the mesh
and its marked edges into the well-known newest-vertex bisection (NVB)
algorithm as described in~\cite[Section~5]{fpw11}.
This refinement procedure then replaces~\eqref{eq:space-enrichment},
i.e. the refined mesh defines the subsequent $\spc P_1$
finite element space.

\begin{algorithm}
    \caption{Edge-based variational adaptivity}
    \label{alg:edge-based-variational-adaptivity}
    \begin{algorithmic}[1]  
            \STATE%
            Prescribe the Dörfler parameter $\theta \in (0, 1)$.
            \STATE%
            Input a finite element mesh $\set T_{N}$ and a
            finite element function $u_{N}^\star \in \spc P_1 (\mathcal{T}_N)$.
            \FOR{$E \in \set E^I_N$}
                \STATE%
                SOLVE~\eqref{eq:local-general-nonlinear-weak-problem} with
                $\spc Y = \text{span}~(\{u^\star_{N}\} \cup \mat \xi_E)$
                for $u^h_{\spc Y} \in \spc Y$.
                \label{alg:edge-based-variational-adaptivity:solve}
                \STATE%
                Compute the corresponding energy decay
                $\Delta \E (u^\star_{N}, \mat \xi_E)$,
                cf.~\eqref{eq:energy-decay}.
                \label{alg:edge-based-variational-adaptivity:energy-diff}
            \ENDFOR
            \STATE%
            MARK a subset of the interior edges $\set E^I_N$
            of minimal cardinality fulfilling the Dörfler marking
            criterion~\eqref{eq:dörfler} for refinement.
            \STATE%
            REFINE the mesh $\mathcal{T}_N$ based on the marked edges
            using newest-vertex bisection (NVB) as described in~\cite[Section~5]{fpw11}
            to obtain the refined mesh $\mathcal{T}_{N+1}$
            and build the corresponding finite element space
            $\spc P_1(\mathcal{T}_{N+1})\supset\spc P_1(\mathcal{T}_N)$.
    \end{algorithmic}
\end{algorithm}

\subsubsection{Computation of energy decay}
\label{sec:eva-nonlinear}

If the energy $\E : \spc V \to \mathbb{R}$
merely satisfies Assumption~\ref{assumption},  
we generally need to resort to a two-fold approximation
of the local energy decay~\eqref{eq:energy-decay}.  
First, the local minimization
problem~\eqref{eq:local-general-nonlinear-weak-problem}  
cannot typically be solved exactly, cf. Remark~\ref{rem:solve-local-system}.
Instead, we compute an approximation  
$\tilde u_{\spc Y} \approx u^h_{\spc Y}$ as follows.  
Assuming that $\E$ is twice Gâteaux differentiable,
we set  
$\delta := \tilde u_{\spc Y} - u_N^\star$
and consider the first-order approximation  
of~\eqref{eq:local-general-nonlinear-weak-problem}, i.e.

\begin{equation}
    \label{eq:first-order-approximation-locally-improved-solution}
    \E''[u^\star_{N}](\delta, y)
    =
    -\E'[u^\star_{N}](y)
    \quad
    \forall y \in \spc Y,
\end{equation}
whose solution can be obtained by solving the corresponding
(low-dimensional) linear system.
Second, given the solution $\tilde u_{\spc Y}$
of~\eqref{eq:first-order-approximation-locally-improved-solution},
the corresponding energy decay would be
obtained by evaluating~\eqref{eq:energy-decay}
with $u^h_{\spc Y}$ replaced by $\tilde u_{\spc Y}$.
However, performing this evaluation for every enrichment set
can become prohibitively expensive
(in the context of finite elements this becomes obvious
due to the global integrals involved in the energy).
Hence, we employ a computationally cheaper approximate predictor.
To this end, applying the main theorem of calculus, we note that it holds
\begin{equation}
    \label{eq:local-energy-reduction-predictor-exact}
    \E(\tilde u_{\spc Y})
    -
    \E(u^\star_{N})
    =
    \E'[u^\star_{N}](\delta)
    +
    \int_0^1
    \left(
        \E'[u^\star_{N} + s \delta](\delta)
        -
        \E'[u^\star_{N}](\delta)
    \right)
    \dd s.
\end{equation}
We use the first term on the right-hand side as a first-order
approximation for the negative energy decay, that is
\begin{equation}
    \label{eq:local-energy-reduction-predictor-approximation-first-order}
    \E(\tilde u_{\spc Y})
    -
    \E(u^\star_{N})
    \approx
    \E'[u^\star_{N}](\delta).
\end{equation}
By \eqref{eq:first-order-approximation-locally-improved-solution}
and the strong convexity of $\E$, we observe that
this predictor is guaranteed to be non-positive.
Incidentally, applying a second-order Taylor expansion, viz.
    \begin{equation}
        \label{eq:energy-expansion-second-order}
        \E'[u^\star_{N} + s\delta](\delta)
        =
        \E'[u^\star_{N}](\delta)
        + s \E''[u^\star_{N}](\delta,\delta) + \mathcal{O}(\delta^3),
    \end{equation}
and exploiting~\eqref{eq:first-order-approximation-locally-improved-solution} with $y=\delta$, 
the integral in~\eqref{eq:local-energy-reduction-predictor-exact} can be expressed by
\[
\int_0^1
    \left(
        \E'[u^\star_{N} + s \delta](\delta)
        -
        \E'[u^\star_{N}](\delta)
    \right)
    \dd s
    =
    \frac12 \E''[u^\star_{N}](\delta,\delta)+\mathcal{O}(\delta^3)
        =
    -\frac12 \E'[u^\star_{N}](\delta)+\mathcal{O}(\delta^3).
\]
Thus, we infer that
    \begin{equation}
        \label{eq:local-energy-reduction-predictor-approximation-second-order}
        \E (\tilde u_{\spc Y}) - \E (u^\star_N)
        =
        \frac{1}{2}\E'[u^\star_{N}](\delta) + \mathcal{O}(\delta^3).
    \end{equation}
Consequently, in the Dörfler marking~\eqref{eq:dörfler}
(where constant factors cancel), using the first-order predictor
\eqref{eq:local-energy-reduction-predictor-approximation-first-order}
is equivalent to neglecting the $\mathcal{O}(\delta^3)$ terms in the above energy-decay
approximation~\eqref{eq:local-energy-reduction-predictor-approximation-second-order}.
In summary, the approximation of the
local energy decay is indeed two-fold:
We approximate the exact minimizer
$\tilde u_{\spc Y} \approx u^h_{\spc Y}$
using~\eqref{eq:first-order-approximation-locally-improved-solution},
and the energy decay corresponding to $\tilde u_{\spc Y}$
using~\eqref{eq:local-energy-reduction-predictor-approximation-first-order}
(or equivalently~\eqref{eq:local-energy-reduction-predictor-approximation-second-order}).

We continue by briefly discussing how the local energy decays can be computed in practice.
In the special case where the energy functional
$\E:\spc V\to\mathbb{R}$ is linear-quadratic
(cf. Section~\ref{sec:simplified-model-problem}),
we may compute both the locally improved solution
and its corresponding energy decay exactly.
To see this, consider, for example, edge-based variational adaptivity
(cf. Section~\ref{sec:eva}).
Then, problem~\eqref{eq:local-general-nonlinear-weak-problem}
is equivalent to the $2\times 2$ linear system
\begin{equation}
    \label{eq:eva-local-problem}
    \begin{pmatrix}
        a\left(u^\star_N, u^\star_N\right) & a\left(u^\star_N, \phi_E\right) \\
        a\left(\phi_E, u^\star_N\right) & a\left(\phi_E, \phi_E\right)
    \end{pmatrix}
    \begin{pmatrix}
        \alpha \\
        \beta
    \end{pmatrix}
    =
    \begin{pmatrix}
        \ell(u^\star_N) \\
        \ell(\phi_E)
    \end{pmatrix},
\end{equation}
such that $u^h_{\spc Y} = \alpha u^\star_N + \beta \phi_E$,
where $\alpha, \beta \in \mathbb{R}$ solve~\eqref{eq:eva-local-problem}.
Note that the (global) terms $a(u^\star_N, u^\star_N)$ and $\ell (u_N^\star)$ are both
independent of individual edges $E\in\set E^I_N$,
and therefore, need to be computed once only.
Moreover, owing to the highly localized support of the basis function $\phi_E$,  
the remaining entries of both the left-hand side matrix and
the right-hand side vector in~\eqref{eq:eva-local-problem},
which depend on $E$, reduce to local integrals.
In addition, we note that the $2 \times 2$
linear system~\eqref{eq:eva-local-problem} admits a closed-form solution.
In this setting, the energy decay~\eqref{eq:energy-decay}
reduces to
\begin{equation}
    \label{eq:eva-energy-decay-per-edge}
    \Delta \E(u_N^\star, \mat \xi_E)
    =
    \E (u^\star_N) - \E(u^h_{\spc Y})
    =
    \frac{(1-\alpha)^2}{2} a(u^\star_{N}, u^\star_{N})
    - \beta (1-\alpha) a(u^\star_{N}, \phi_E)
    +\frac{\beta^2}{2} a(\phi_E, \phi_E).
\end{equation}
We observe that the computation of the energy decay per
edge~\eqref{eq:eva-energy-decay-per-edge} is highly parallelizable,
and essentially involves only two global evaluations for the entire mesh,
i.e. the quantities $a(u^\star_{N}, u^\star_{N})$ and $\ell(u^\star_{N})$.
    
    For a general energy functional $\E : \spc V \to \mathbb{R}$
    satisfying Assumption~\ref{assumption},
    problem~\eqref{eq:first-order-approximation-locally-improved-solution}
    inherits the same structure as~\eqref{eq:eva-local-problem},
    where the size of the system is determined by the number of elements
    in the enrichment set at hand. 

\begin{remark}
    For a linear-quadratic energy functional $\E$,
    the approximations~\eqref{eq:first-order-approximation-locally-improved-solution}
    and~\eqref{eq:local-energy-reduction-predictor-approximation-first-order}
    yield an equivalent marking strategy as performing the exact computations.
    To see this, consider the following.
    First, the linearized problem
    \eqref{eq:first-order-approximation-locally-improved-solution}
    indeed reduces to \eqref{eq:eva-local-problem} and hence, yields
    the exact minimizer $\tilde u_{\spc Y} = u^h_{\spc Y} \in \spc Y$.
    Second, as $\mathcal{O}(\delta^3)$ terms vanish,
    expression~\eqref{eq:local-energy-reduction-predictor-approximation-second-order}
    tells us that, up to a factor $\nicefrac{1}{2}$, the first order
    predictor corresponds to the exact energy decay.
    Hence, in the Dörfler marking strategy~\eqref{eq:dörfler},
    the resulting set of marked edges agree
    for the approximate and exact computations.
\end{remark}

\begin{remark}
    At first glance, our approach might appear reminiscent of the classical
    Zienkiewicz-Zhu (ZZ) error estimator (see, e.g.~\cite[Chapter 4.6.1]{AO00}).
    We emphasize, however, that the ZZ estimator aims at an a posteriori bound for the error,
    whilst our method is based on computing
    local \emph{energy reduction indicators}.
    Furthermore, the ZZ estimator is recovery-based,
    i.e., it estimates the local error
    by suitably reconstructing the gradient of the current approximation,
    which, in contrast to the variational adaptivity concept,
    is computed using local information only;
    indeed, crucially in our framework, a potentially improved
    approximation of the given solution
    is computed using a
    \emph{global mode} and a local
    (low-dimensional) enrichment of the approximation space at
    any edge $E\in\mathcal{E}^I_N$,
    cf.~\eqref{eq:first-order-approximation-locally-improved-solution}.
\end{remark}

\subsection{Fully iterative solver}
\label{sec:solver}

All the machinery needed to derive a fully iterative adaptive solver
for problem~\eqref{eq:full-space-problem} with $\spc V = \Hone$
is now in place and outlined in Algorithm~\ref{alg:full-solver}.
Roughly speaking, our approach consists of two intertwined strategies:
\begin{enumerate}[1.]

    \item%
    On a given finite dimensional subspace
    $\spc P_1(\mathcal{T}_N) \subset \Hone$,
    we perform (nonlinear) CG iterations until one of the stopping
    criteria described in Section~\ref{sec:stopping-criteria} is met,
    thereby yielding a final iterate
    $u^\star_{N} \in \spc P_1(\mathcal{T}_N)$
    (Line~\ref{alg:full-solver:cg} in Algorithm~\ref{alg:full-solver}).

    \item%
    We use the final iterate
    $u^\star_N \in \spc P_1(\mathcal{T}_N)$
    to compute local energy reduction indicators within the framework
    of edge-based variational adaptivity
    as described in Section~\ref{sec:eva}.
    Combined with the Dörfler marking strategy, these indicators
    determine a subset of interior edges to be refined.
    To guarantee that the refinement process
    preserves mesh shape regularity and avoids degenerating elements,
    we employ NVB as outlined in~\cite[Section~5]{fpw11}
    to refine the mesh and to finally construct the subsequent finite
    element space $\spc P_1(\mathcal{T}_{N+1})$
    (Line~\ref{alg:full-solver:eva} in Algorithm~\ref{alg:full-solver}).

\end{enumerate}
After these two steps have been completed,
an initial guess $u^0_{N+1} \in \spc P_1(\mathcal{T}_{N+1})$
is constructed by linearly interpolating the
final iterate $u^\star_N$ onto the refined mesh
(Line~\ref{alg:full-solver:embedding}
in Algorithm~\ref{alg:full-solver}).
Finally, the procedure is restarted
and repeated until a specified maximum
number of degrees of freedom is reached or, alternatively, 
until the relative energy decay of the sequence of the 
final iterates on each space, i.e. $\{\E(u^\star_N)\}_N$,
falls below a predefined threshold.

\begin{algorithm}
    \begin{algorithmic}[1]  
            \STATE%
            Prescribe the Dörfler parameter $\theta \in (0, 1)$,
            a maximum number of degrees of freedom $N_{\text{max}}$,
            choose one of the stopping criteria described in
            Section~\ref{sec:stopping-criteria},
            and fix the control parameter set
            needed for the chosen stopping criterion.
            \STATE%
            Input an initial mesh $\set T_{N}$,
            and an initial guess $u^0_{N} \in \spc P_1(\set T_{N})$.
            \WHILE{$ \text{dim}~\spc P_1(\mathcal{T}_N) \leq N_{\text{max}}$}
                \STATE%
                \label{alg:full-solver:cg}
                Initialize (nonlinear) CG iterations with $u^0_{N}$
                and stop the iteration at $u^\star_{N}$
                as soon as the chosen stopping criterion is met.
                \STATE%
                \label{alg:full-solver:eva}
                Based on $u^\star_{N}$ and the current mesh $\set T_{N}$,
                perform the edge-based variational adaptivity
                procedure as indicated in
                Algorithm~\ref{alg:edge-based-variational-adaptivity}
                to obtain a refined mesh $\set T_{N+1}$ and
                the corresponding enriched
                finite element space $\spc P_1(\mathcal{T}_{N+1})$.
                \STATE%
                \label{alg:full-solver:embedding}
                Define $u^0_{N+1} := u^\star_{N}$,
                set $N \leftarrow N+1$,
                and reset all control parameters from the chosen stopping criterion
                that might have been changed during the CG iterations.
            \ENDWHILE
    \end{algorithmic}
    \caption{Fully iterative edge-based variational adaptivity $\mathbb{P}_1$-FEM}
    \label{alg:full-solver}
\end{algorithm}

\section{Numerical Experiments}
\label{sec:numerical-experiments}
In this Section, we perform a series of numerical tests for
Algorithm~\ref{alg:full-solver}. Namely,
on an open and bounded polygonal domain
$\Omega \subset \R^2$, which will be either
chosen to be the \emph{L-shaped domain}
$\Omega_{\text{L}} := (-1, 1)^2 \setminus ([0,1] \times [-1, 0])$,
or the \emph{square domain} $\Omega_{\text{S}} := (0, 1)^2$, we consider the numerical
solution of the semilinear elliptic diffusion-reaction
boundary value problem (BVP):
\begin{equation}
    \label{eq:non-linear-pde}
    \begin{aligned}
        - \nabla \cdot
        \left(
            \mat A (\mat x) \nabla u(\mat x)
        \right)
        + \phi(u(\mat x))
        & =
        f(\mat x),
        &\mat x \in & \Omega,
        \\
        u(\mat x) &= 0,
        & \mat x \in & \partial \Omega,
    \end{aligned}
\end{equation}
where $\phi \in \mathrm{C}^0(\mathbb{R})$ is non-decreasing,
$f\in \mathrm{L^2}(\Omega)$,
and $\mat A : \Omega \to \spc R^{2\times 2}$
is a piecewise constant symmetric diffusion matrix that
is uniformly positive definite,
i.e. there exists a constant $C>0$ such that
\begin{equation}
    \label{eq:uniform-ellipticity}
    \forall \mat x \in \Omega
    \,:\,
    \inf_{\substack{\mat\xi\in\spc\R^2\\\mat\xi\neq\mat 0}}
    \frac{
        \mat\xi^\intercal\mat A (\mat x)\mat\xi
    }{
        \mat\xi^\intercal\mat\xi
    }
    \ge C.
\end{equation}
Then, the BVP ~\eqref{eq:non-linear-pde} admits a unique solution $u \in \Hone$~\cite[chapter 4]{glowinski84};
in our numerical tests below,
the data will be specified individually for each experiment.

We define the energy functional
$\E : \Hone \to \mathbb{R}$
corresponding to~\eqref{eq:non-linear-pde} as
\begin{equation}
    \label{eq:non-linear-energy-functional-for-pde}
    \E (v) :=
    \frac{1}{2}
    \int_{\Omega} \mat A (\mat x) \nabla v (\mat x) \cdot \nabla v(\mat x) ~\dd \mat x
    +
    \int_{\Omega} 
    \Phi(v(\mat x))\,
    \dd \mat x 
    -
    \int_{\Omega} f(\mat x) v(\mat x) ~\dd \mat x,
\end{equation}
with $\Phi'=\phi$.
Indeed, $\E$ satisfies Assumption~\ref{assumption} and hence,
there exists a unique minimizer $u \in \Hone$ of $\E$ that
is also the weak solution of~\eqref{eq:non-linear-pde},
cf. Section~\ref{sec:introduction}.
Furthermore, the gradient term in the energy induces an energy norm on $\Hone$, viz.
\begin{equation}
    \label{eq:energy-norm-pde}
    \|u\|_a^2
    :=
    \int_\Omega \mat A (\mat x)
    \nabla u (\mat x) \cdot \nabla u (\mat x)
    ~\dd \mat x.
\end{equation}

In order to ensure that the
respective iterative solver employed
is able to sufficiently progress
before convergence towards the (exact) Galerkin solution is assessed,
a small number of iterations, denoted by $n_{\text{min}}$,
is enforced on each subspace
before evaluating the respective stopping criterion.
In particular, when using the energy tail-off stopping
criterion~\eqref{eq:energy-tail-off-stopping-criterion},
the condition is reformulated as
\begin{equation}\label{eq:critmod}
    \E(u^{n}_{N}) - \E(u^{n-1}_{N})
    <
    \alpha_{\E}
    \frac{
        \E(u^{n_{\text{min}}}_{N}) - \E(u^{n}_{N})
        }{
        n - n_{\text{min}}
    },
    \quad
    \text{once } n > n_{\text{min}}.
\end{equation}
Moreover, owing to the adaptive adjustment of the delay parameter $d$, 
see Remark~\ref{remark:stopping-criteria}~(a),
we emphasize that its value may vary during the (nonlinear)
CG iterations on a given mesh, however,
is always reset to $d_{\text{init}}$ after each mesh refinement,
cf. Line~\ref{alg:full-solver:embedding}
in Algorithm~\ref{alg:full-solver}.

In the performance plots, which we display for each numerical test, the stopping criteria are denoted as follows: criterion~\eqref{eq:energy-tail-off-stopping-criterion} is referred to as
\emph{energy tail-off},
criterion~\eqref{eq:energy-arioli-stopping-criterion-abs} as
\emph{relative energy reduction},
and criterion~\eqref{eq:stopping-criterion-default} as
\emph{default}.
In addition, the following control parameters
are available to be chosen in each of the experiments:
\begin{itemize}
    \item%
    The Dörfler parameter $\theta \in (0, 1)$
    in~\eqref{eq:dörfler}
    (set to $\theta = 0.5$ in \emph{all} experiments),

    \item%
    the control parameter $\alpha_{\E}$
    of the energy tail-off stopping
    criterion~\eqref{eq:critmod},

    \item%
    the control parameter $\alpha_{\gamma}$
    of the relative energy decay stopping
    criterion~\eqref{eq:energy-arioli-stopping-criterion-abs},
 
    \item%
    the initial delay $d_{\text{init}} \geq 1$,
    fixed delay increment $\Delta d \geq 1$,
    and control parameter $\tau > 1$
    for the adaptive delay parameter
    corresponding to~\eqref{eq:adaptive-delay}
    and~\eqref{eq:energy-arioli-stopping-criterion-abs},
    
    \item%
    the minimum number of iterations $n_{\text{min}}$
    performed on each mesh before any stopping criterion is evaluated,
    
    \item%
    the batch size $n_{\text{batch}} \geq 1$ specifying the
    number of iterations between consecutive evaluations
    of the respective stopping criterion,
    
    \item%
    and the threshold $g_{\text{tol}}$ of the default stopping
    criterion~\eqref{eq:stopping-criterion-default}
    (necessary only in the nonlinear experiments and set to
    $g_{\text{tol}}=10^{-8}$ in \emph{all} the nonlinear experiments).
\end{itemize}
As a visual reference, in all of our convergence plots,
we will display a dashed line with slope
$-\nicefrac{1}{2}$ (or $-1$)
corresponding to the optimal convergence rate for the energy norm
(or its square)
with respect to the number of degrees of freedom
in $\spc P_1$-FEM for second order elliptic problems in two dimensions.

Any use of an existing implementation in
\verb|scipy|~\cite{scipy}
is based on
version~\verb|1.13.1|,
and all calculations were performed on the HPC cluster UBELIX
(\href{https://www.id.unibe.ch/hpc}{id.unibe.ch/hpc}) at the University of Bern.

The remainder of this Section is divided into two main parts:
in Section~\ref{sec:nonlinear-experiments}, we consider the case
where the reaction $\phi$ is nonlinear, whereas in
Section~\ref{sec:linear-experiments} we focus on the case
where $\phi$ is linear.

\subsection{Nonlinear Experiments}
\label{sec:nonlinear-experiments}
In this Section, we assess Algorithm~\ref{alg:full-solver}
for the numerical minimization of the energy
functional~\eqref{eq:non-linear-energy-functional-for-pde}
in the case where $\phi$ is nonlinear.
As motivated in Section~\ref{sec:extended-applicability},
the iterative solver employed on each subspace is
chosen to be the Polak--Ribi\`ere type
nonlinear CG method implemented in
\verb|scipy.optimize.fmin_cg|.
In all experiments below,
the computational domain is
the L-shaped domain $\Omega_{\text{L}}$.
The specific choices of the control parameters
are presented in the respective experiments below.
We consider four distinct boundary value problems,
where, for each problem, we report results for both
stopping criteria presented in Section~\ref{sec:stopping-criteria}
and, as a reference, for the default
stopping criterion~\eqref{eq:stopping-criterion-default} implemented in
\verb|scipy.optimize.fmin_cg|.

\subsubsection*{Computational Setup}
Recall from Section~\ref{sec:eva-nonlinear}
that we assume $\E$ to be twice Gâteaux differentiable
such that the locally improved solution may be obtained by
solving~\eqref{eq:first-order-approximation-locally-improved-solution}.
Here, this translates into assuming
the nonlinearity $\phi$ to be continuously differentiable,
i.e. $\phi \in \mathrm{C}^1(\mathbb R)$.

Note that for
Experiments~\ref{experiment-nonlinear-1}--\ref{experiment-nonlinear-3},
the exact full-space solution $u \in \Hone$ is unknown.
Consequently, in the plots displaying the energy difference between
the final iterates and the energy of the full-space solution,
the latter is approximated by the energy of a
reference solution $\tilde u \approx u$
computed on a graded mesh towards the origin in the L-shaped domain~$\Omega_{\mathrm{L}}$, where all of our test problems are expected to exhibit a typical elliptic corner singularity; the strategy employed to create the graded mesh and the solver
used to obtain the reference solutions thereon are both outlined
in~\cite{spicher25}\footnote{The authors gratefully acknowledge Florian Spicher for providing the corresponding code.}.
In contrast, Experiment~\ref{experiment-nonlinear-4} is based on a
manufactured solution, and therefore, we are able to compute the
energy-norm error between the exact full-space
solution and the final iterates.

All experiments below were initialized as follows.
First, the mesh of the underlying domain was uniformly refined
until a sufficiently large number of degrees of freedom was obtained.
Second, the nonlinear CG method
\verb|scipy.optimize.fmin_cg|
was applied using the default
stopping criterion~\eqref{eq:stopping-criterion-default}
with its default parameter settings.
Finally, the mesh was refined using edge-based variational adaptivity
as described in Algorithm~\ref{alg:edge-based-variational-adaptivity},
and the subsequent experiment was initialized on this refined mesh,
with the initial guess given by the final iterate, embedded into the new space.
All numerical integrations have been carried out using a
quadrature rule 
with degree of exactness~$6$.

Finally, we point out that, in all experiments below,
we employ the default stopping criterion
with $g_{\text{tol}} = 10^{-8}$, which is three orders of magnitude
smaller than the value suggested by the default settings.
Consequently, the corresponding values in the plots below
may be regarded as good approximations of the discretization
energy (norm) error and serve as an indicator of the
effectiveness of the proposed refinement strategy.

\begin{experiment}
    \label{experiment-nonlinear-1}
    In this experiment, we consider the
    BVP~\eqref{eq:non-linear-pde} with
    the constant diffusion matrix $\mat A := \Id$,
    the nonlinearity $\phi(u) := u^3$ (corresponding to the potential $\Phi(u)=\tfrac14u^4$),
    and the constant right-hand side function $f := 1$.
    The convergence plot in Figure~\ref{fig:nonlinear-1}
    displays results for Algorithm~\ref{alg:full-solver}
    with the following choices of the control parameters
    for the corresponding stopping criteria.
    Relative energy reduction:
    $\alpha_\gamma = 0.1$,
    $n_{\text{min}} = 10$,
    $\tau = 1.01$, $d_{\text{init}}=10$,
    $\Delta d = 5$.
    Energy tail-off:
    $\alpha_{\E} = 0.01$,
    $n_{\text{min}} = 10$,
    $n_{\text{batch}} = 5$.
    Default:
    $n_{\text{min}} = 10$,
    $g_{\text{tol}} = 10^{-8}$.

    In Figure~\ref{fig:nonlinear-1},
    we observe that both the relative energy reduction stopping
    criterion~\eqref{eq:energy-arioli-stopping-criterion-abs}
    and the energy tail-off stopping
    criterion~\eqref{eq:energy-tail-off-stopping-criterion}
    exhibit optimal convergence behavior,
    although the corresponding data points are located
    about half an order of magnitude above those obtained with the
    default stopping
    criterion~\eqref{eq:stopping-criterion-default}.
    Finally, we observe that, compared to the default stopping criterion,
    our new energy-based criteria require considerably fewer
    space enrichment steps.

    \begin{figure}
        \centering
        \includegraphics[width=0.7\textwidth]{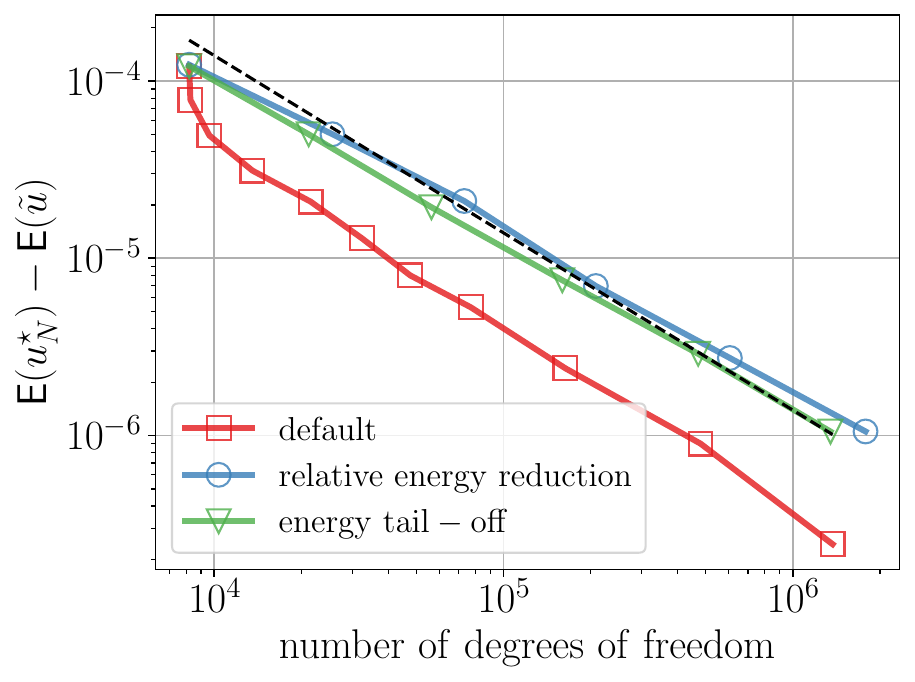} 
        \caption{
            Convergence plot for
            Experiment~\ref{experiment-nonlinear-1},
            obtained with Algorithm~\ref{alg:full-solver}
            and three different stopping criteria.
            The black dashed line represents the optimal
            convergence rate $\mathcal{O}(n_{\text{DOF}}^{-1})$.
        }
        \label{fig:nonlinear-1}
    \end{figure}

\end{experiment}

\begin{experiment}
    \label{experiment-nonlinear-2}
    In this experiment, we consider the
    BVP~\eqref{eq:non-linear-pde} with
    the constant diffusion matrix $\mat A := \Id$,
    the nonlinearity $\phi(u) := |u| u$ (corresponding to $\Phi(u)=\tfrac13|u|^3$),
    and the constant right-hand side function $f := 1$.
    The convergence plot in Figure~\ref{fig:nonlinear-1}
    displays results for Algorithm~\ref{alg:full-solver}
    with the following choices of the control parameters
    for the corresponding stopping criteria.
    Relative energy reduction:
    $\alpha_\gamma = 0.1$,
    $n_{\text{min}} = 10$,
    $\tau = 1.01$, $d_{\text{init}}=10$,
    $\Delta d = 5$.
    Energy tail-off:
    $\alpha_{\E} = 0.01$,
    $n_{\text{min}} = 10$,
    $n_{\text{batch}} = 5$.
    Default:
    $n_{\text{min}} = 10$,
    $g_{\text{tol}} = 10^{-8}$.

    In Figure~\ref{fig:nonlinear-2},
    we observe that the energy tail-off stopping
    criterion~\eqref{eq:energy-tail-off-stopping-criterion}
    performs considerably better than the relative
    energy reduction stopping
    criterion~\eqref{eq:energy-arioli-stopping-criterion-abs}.
    The slope corresponding to the latter even appears to slightly level off asymptotically.
    Nevertheless, the results obtained with both the default stopping criterion
    and the energy tail-off stopping criterion closely follow the dashed line,
    indicating optimal convergence behavior in both cases.
    We therefore conclude that the adaptive edge-based refinement strategy
    is capable of resolving the corner singularity at the origin
    not only when using the final iterate produced by the default stopping criterion,
    but also when using the final iterate obtained with the energy tail-off
    stopping criterion.

    \begin{figure}
        \centering
        \includegraphics[width=0.7\textwidth]{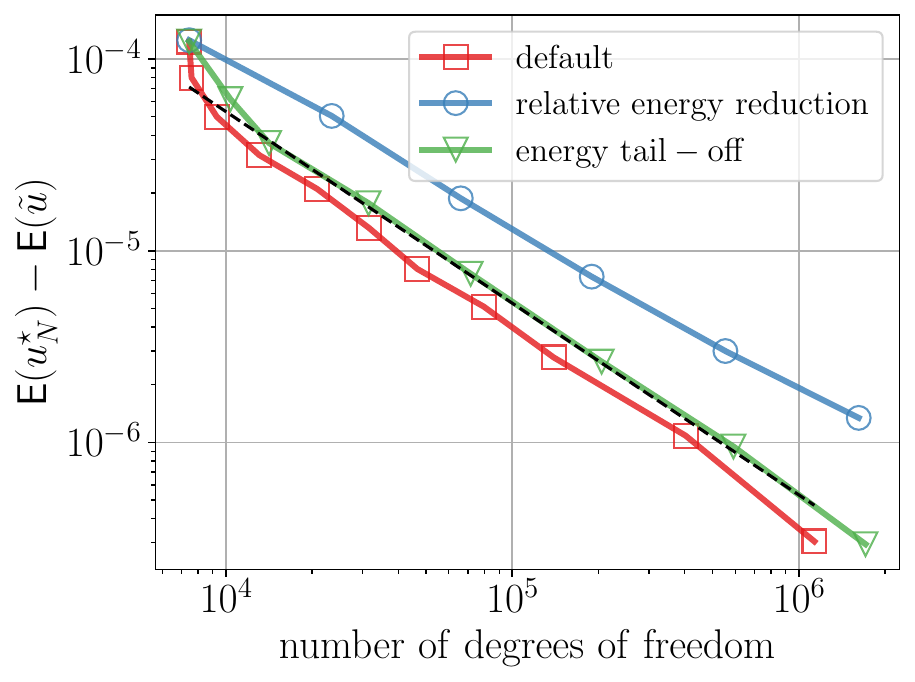} 
        \caption{
            Convergence plot for
            Experiment~\ref{experiment-nonlinear-2},
            obtained with Algorithm~\ref{alg:full-solver}
            and three different stopping criteria.
            The black dashed line represents the optimal
            convergence rate $\mathcal{O}(n_{\text{DOF}}^{-1})$.
        }
        \label{fig:nonlinear-2}
    \end{figure}

    
\end{experiment}

\begin{experiment}
    \label{experiment-nonlinear-3}
    In this experiment, we consider the
    BVP~\eqref{eq:non-linear-pde} with
    the constant diffusion matrix $\mat A := \Id$,
    the nonlinearity $\phi(u) := \exp (u) - 1$,
    and the constant right-hand side function $f := 1$.
    The convergence plot in Figure~\ref{fig:nonlinear-1}
    displays results for Algorithm~\ref{alg:full-solver}
    with the following choices of the control parameters
    for the corresponding stopping criteria.
    Relative energy reduction:
    $\alpha_\gamma = 0.01$,
    $n_{\text{min}} = 10$,
    $\tau = 1.001$, $d_{\text{init}}=50$,
    $\Delta d = 5$.
    Energy tail-off:
    $\alpha_{\E} = 0.01$,
    $n_{\text{min}} = 10$,
    $n_{\text{batch}} = 5$.
    Default:
    $n_{\text{min}} = 10$,
    $g_{\text{tol}} = 10^{-8}$.

    Surprisingly, Figure~\ref{fig:nonlinear-3} shows almost identical
    results as we obtained for Experiment~\ref{experiment-nonlinear-1}.
    Therefore, we again conclude that both
    the relative energy reduction
    criterion~\eqref{eq:energy-arioli-stopping-criterion-abs}
    and energy tail-off stopping
    criterion~\eqref{eq:energy-tail-off-stopping-criterion}
    yield optimal convergence behavior,
    although they are both located about half an order of magnitude
    above the slope corresponding to the default stopping
    criterion~\eqref{eq:stopping-criterion-default}.

    \begin{figure}
        \centering
        \includegraphics[width=0.7\textwidth]{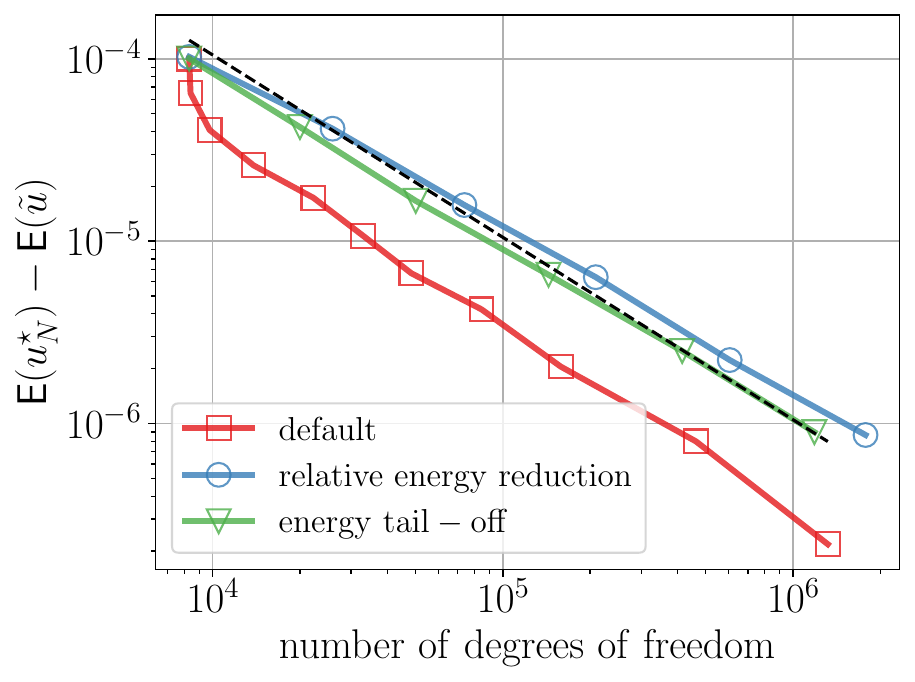} 
        \caption{
            Convergence plot for
            Experiment~\ref{experiment-nonlinear-3},
            obtained with Algorithm~\ref{alg:full-solver}
            and three different stopping criteria.
            The black dashed line represents the optimal
            convergence rate $\mathcal{O}(n_{\text{DOF}}^{-1})$.
        }
        \label{fig:nonlinear-3}
    \end{figure}

\end{experiment}

\begin{experiment}
    \label{experiment-nonlinear-4}
    In this experiment, we consider the
    BVP~\eqref{eq:non-linear-pde} with
    the constant diffusion matrix $\mat A := \Id$,
    and the nonlinearity $\phi(u) := u^3$.
    Further, we choose the right-hand side function $f$
    such that the exact solution is given by the singular function
    \begin{equation}
        u(x, y) := 2 r^{-4/3} xy (1-x^2)(1-y^2),
    \end{equation}
    where $r:= \sqrt{x^2 + y^2}$ is the radial distance to the origin.
    The convergence plot in Figure~\ref{fig:nonlinear-1}
    displays results for Algorithm~\ref{alg:full-solver}
    with the following choices of the control parameters
    for the corresponding stopping criteria.
    Relative energy reduction:
    $\alpha_\gamma = 0.1$,
    $n_{\text{min}} = 10$,
    $\tau = 1.01$, $d_{\text{init}}=10$,
    $\Delta d = 5$.
    Energy tail-off:
    $\alpha_{\E} = 0.01$,
    $n_{\text{min}} = 10$,
    $n_{\text{batch}} = 5$.
    Default:
    $n_{\text{min}} = 10$,
    $g_{\text{tol}} = 10^{-8}$.

    In Figure~\ref{fig:nonlinear-4},
    we clearly observe that the results obtained with the
    energy tail-off stopping
    criterion~\eqref{eq:energy-tail-off-stopping-criterion}
    are almost identical to those obtained with the default stopping
    criterion~\eqref{eq:stopping-criterion-default}.
    In both cases,
    the corresponding slopes closely follow the dashed line,
    indicating optimal convergence behavior.
    Finally, as already observed in
    Experiment~\ref{experiment-nonlinear-2},
    the slope corresponding to the relative energy decay stopping
    criterion~\eqref{eq:energy-arioli-stopping-criterion-abs}
    appears to level off slightly once the number of degrees of freedom
    exceeds a magnitude of~$10^6$.

    \begin{figure}
        \centering
        \includegraphics[width=0.7\textwidth]{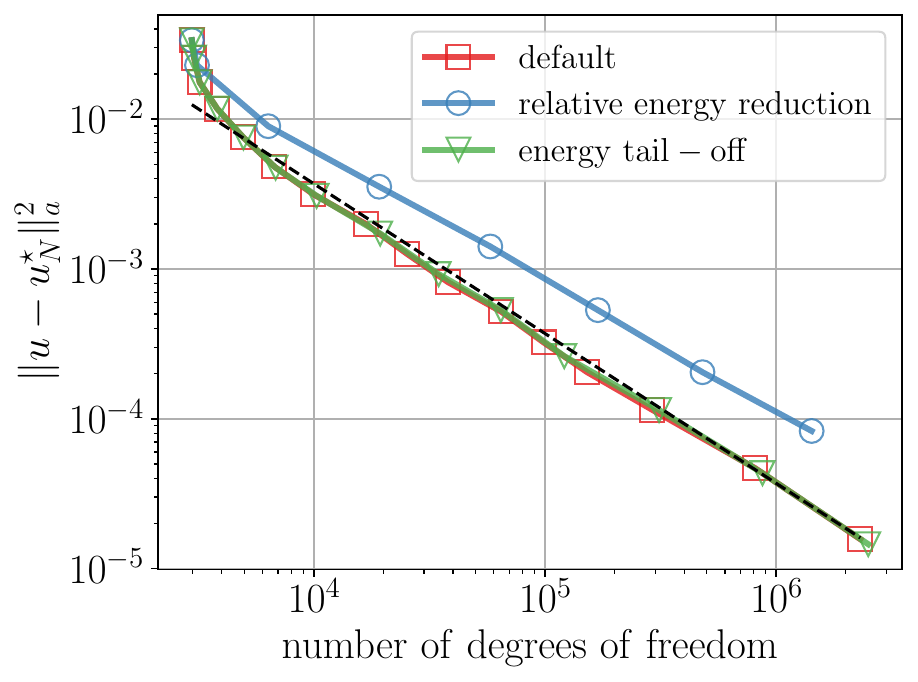} 
        \caption{
            Convergence plot for
            Experiment~\ref{experiment-nonlinear-4},
            obtained with Algorithm~\ref{alg:full-solver}
            and three different stopping criteria.
            The black dashed line represents the optimal
            convergence rate $\mathcal{O}(n_{\text{DOF}}^{-1})$.
        }
        \label{fig:nonlinear-4}
    \end{figure}

\end{experiment}

\subsection{Linear Experiments}
\label{sec:linear-experiments}
In this section, we assess Algorithm~\ref{alg:full-solver}
employed for the numerical solution of~\eqref{eq:non-linear-pde},
where we set
$\phi(u) = c u$, with $c\in \{0, 1\}$ constant.
In this setting,
the energy~\eqref{eq:non-linear-energy-functional-for-pde}
reduces to
\begin{equation}
    \E (v)
    =
    \frac{1}{2}
    \int_\Omega \mat A (\mat x) \nabla v(\mat x) \cdot \nabla v(\mat x)
    ~\dd \mat x
    +
    \frac{c}{2} \int_\Omega v(\mat x)^2~\dd \mat x
    -
    \int_\Omega f(\mat x) v(\mat x)
    ~\dd \mat x,
\end{equation}
possessing the same structure as the energy
in~\eqref{eq:energy-linear-quadratic}.
As motivated in Section~\ref{sec:iterative-solver}, 
the iterative solver, which is employed in this case, is
the standard linear CG algorithm
provided in~\verb|scipy.sparse.linalg.cg|.

We consider four distinct boundary value problems.
For each problem, we present numerical results corresponding
to both stopping criteria introduced in Section~\ref{sec:stopping-criteria}.
The following control parameters are used throughout all experiments:
$\theta = 0.5$,
$\alpha_{\E}=0.1$,
$\tau = 1.01$,
$d_{\text{init}} = 10$,
$\Delta d = 10$,
and $n_{\text{batch}} = 2$.
All additional control parameters, i.e.
$n_{\text{min}}$ and $\alpha_\gamma$,
are specifically selected for each experiment,
and their choice will be detailed in the respective numerical tests below.

In all experiments, the plots should be read as follows.
The line plots with square and circular markers
correspond to the vertical axis on the left,
representing the energy norm of the errors.
In the legend, $u^\star_{N} \in \spc V_{N}$
denotes the final iterate at which the iteration is terminated on the present space, and for which the energy norm of the error is evaluated and plotted.
The line plot with triangular markers corresponds to the vertical
axis on the right and depicts the number of iterations $n(N)$
required on each subspace $\spc V_{N}$ before the corresponding
stopping criterion terminates the iteration,
i.e. we have
$u^{n(N)}_{N} = u^\star_N$.    

\subsubsection*{Computational setup}

As the energy is linear-quadratic in this case,
the local energy decays in
line~\ref{alg:edge-based-variational-adaptivity:energy-diff}
of Algorithm~\ref{alg:edge-based-variational-adaptivity}
are computed with~\eqref{eq:eva-energy-decay-per-edge}.
In the plots corresponding to the experiments below,
we report results for the energy norm of both the
iteration error $\|u^h_N - u_N^\star\|_a$
and the overall error $\|u^\star_N - u\|_a$.
To compute the energy norm of the iteration error,
we compute the Galerkin solution $u^h_N$
on each mesh using a direct solver.
However, to compute the energy norm of the
total error $\|u - u^\star_N\|_a$,
we must resort to an approximate procedure, as for all problems
considered, the exact full-space solution is unknown.
To this end, we exploit~\eqref{eq:galerkin-orthogonality-on-general-subspace},
from which we deduce the useful identity
\begin{equation}
    \|u - u_{N}^\star\|_a^2
    =
    2 \E (u_{N}^\star) + \|u\|_a^2.
\end{equation}
The term $\|u\|_a^2$ is then approximated by a reference value
computed on a strongly refined uniform mesh,
and its numerical value is provided within the description of each respective experiment below.
We note, however, that employing a direct solver to compute the reference value on the final mesh was unfeasible due to memory limitations of the available hardware; therefore, we have resorted to an iterative procedure instead.
More precisely, starting from a coarse initial mesh, we
perform a sequence of uniform mesh refinements;
on each mesh, the CG method is applied with a relative tolerance of
$\epsilon_{\text{rel}} = 10^{-5}$, that is, until it holds the bound
\begin{equation}
    \|\mat b - \mat A \mat x_n \|_2
    \leq
    \epsilon_{\text{rel}} \| \mat b\|_2,
\end{equation}
where $\|\cdot\|_2$ denotes the standard Euclidean norm.
The final approximation on each mesh is recycled as initial guess
on the corresponding refined mesh,
and the final approximation on the finest mesh is used
to compute the reference value.
For Experiments~\ref{experiment-linear-1} and \ref{experiment-linear-2},
this process was continued until the number of
degrees of freedom reached a magnitude of
$\text{dim} \spc V_N \approx 0.30\cdot 10^8$,
while for Experiment~\ref{experiment-linear-3}, it was stopped at
$\text{dim} \spc V_N \approx 0.23\cdot 10^{8}$.

We emphasize that all numerical integrations performed in assembling
the linear systems in our experiments are carried out exactly.
Indeed, for Experiments~\ref{experiment-linear-1}, \ref{experiment-linear-2},
and \ref{experiment-linear-4}, this is straightforward since the
problem data is constant across the entire domain; in Experiment~\ref{experiment-linear-3}, exact integration is ensured
by choosing the initial mesh so that its edges
align with the interface boundaries of each subdomain,
ensuring the data is constant on each element.

\begin{experiment}
    \label{experiment-linear-1}
    In this experiment, we consider the
    BVP~\eqref{eq:non-linear-pde} with
    an anisotropic constant diffusion matrix
    \begin{equation}
        \mat A(\mat x) =
        \begin{pmatrix}
            1 & 0
            \\
            0 & 10^{-2}
        \end{pmatrix},
    \end{equation}
    and a constant reaction constant $c = 1$
    on the square domain $\Omega_{\text{S}}$.
    Due to the small diffusion in the second coordinate direction,  
    we expect boundary layers on the top and bottom boundary of~$\Omega_{\text{S}}$ to form in the true solution.
    To make sure that the convergence is close to optimal
    already for comparatively small numbers of degrees of freedom,
    these boundary layers must be resolved rapidly by the adaptive refinement strategy.
    The energy norm of the true solution 
    is approximated by the reference value
    \begin{equation}
        \|u\|_a^2
        \approx
        \underline{0.071218}57719182778,
    \end{equation}
    where, throughout, the underlined part indicates the stable digits in the numerical approximations.
    From the convergence plots in Figure~\ref{fig:Exp1}, we observe that the overall error
    follows the optimal convergence rate for both stopping criteria proposed in Section~\ref{sec:stopping-criteria}.
    In both cases, the iteration error remains relatively small compared
    to the total error, thus leading to the conclusion that the discretization
    error dominates throughout the experiment.
    Moreover, for either stopping criterion,
    the number of iterations exhibits a single distinct (comparatively high) outlier, which, indicated by the red line with square markers,
    appears to ensure that the CG iteration error
    maintains a certain gap from the total error.
    Finally, we emphasize that, using our approach,
    the number of iterations per mesh was reduced by roughly an order of magnitude
    in comparison with a solution strategy
    based on uniform mesh refinements combined with the default residual-based
    stopping criterion of the CG algorithm
    \verb|scipy.sparse.linalg.cg|;
    this observation holds true in all of our tests, for which
    we computed the reference value of the energy norm of the full space solution
    using the CG algorithm with uniform mesh refinements,
    i.e. for the current experiment, as well as for the subsequent Experiments~\ref{experiment-linear-2} and~\ref{experiment-linear-3}.

    \begin{figure}
        \centering
        \includegraphics[width=0.7\textwidth]{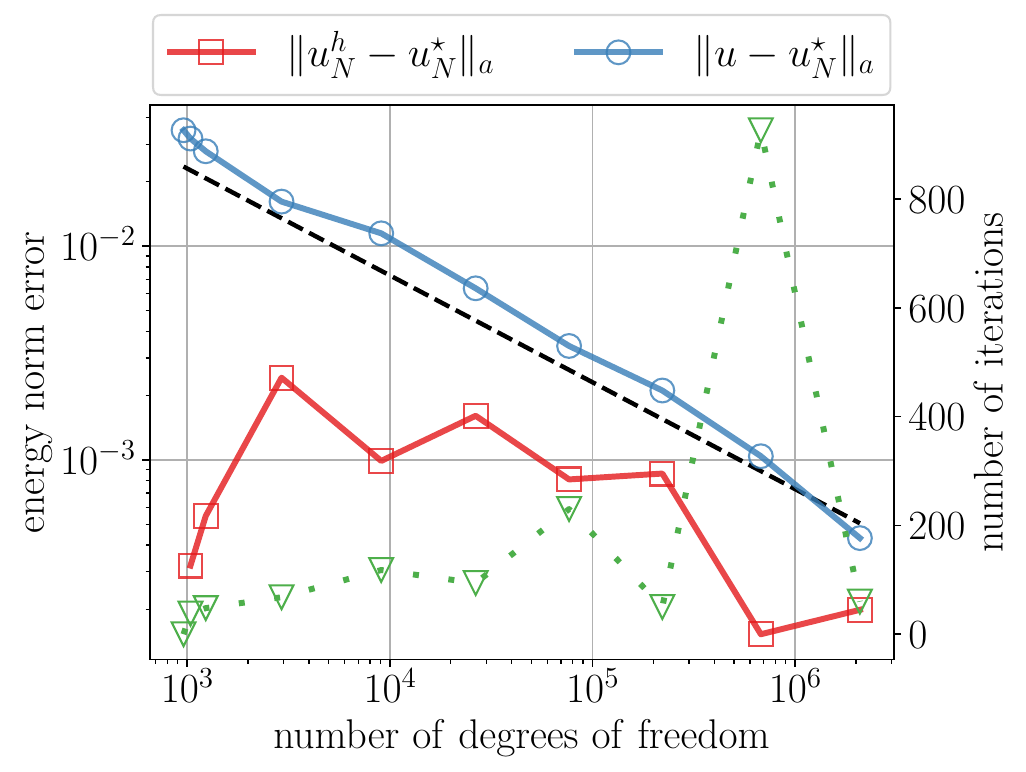} 
        \includegraphics[width=0.7\textwidth]{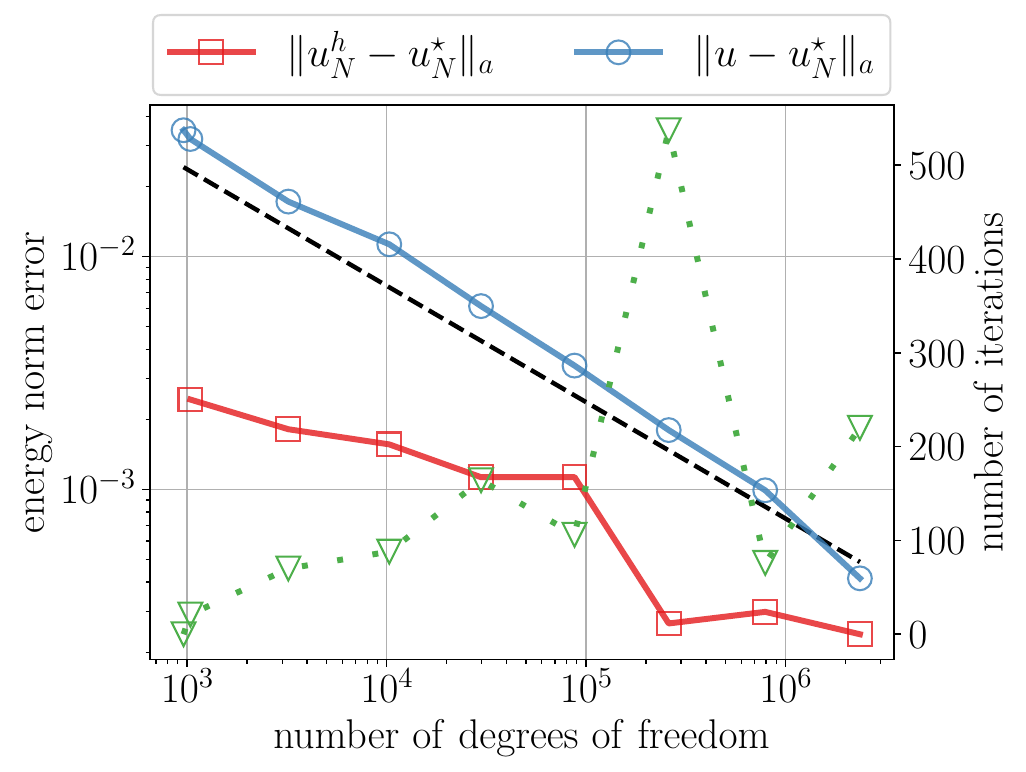} 
        \caption{
            Experiment~\ref{experiment-linear-1}.
            Top: 
            energy tail-off
            stopping
            criterion~\eqref{eq:energy-tail-off-stopping-criterion} 
            with $n_{\text{min}} = 10$.
            Bottom:
            relative energy reduction
            stopping
            criterion~\eqref{eq:energy-arioli-stopping-criterion} 
            with
            $n_{\text{min}} = 10$,
            and $\alpha_\gamma = 0.1$.
        }
        \label{fig:Exp1}
    \end{figure}
\end{experiment}

\begin{experiment}
    \label{experiment-linear-2}
    In this experiment, we consider
    the BVP~\eqref{eq:non-linear-pde}
    with a constant diffusion matrix
    \begin{equation}
        \mat A(\mat x) =
        10^{-2} \mat 1_{2\times 2},
    \end{equation}
    where $\mat 1_{2\times 2}$ denotes the
    $2 \times 2$ identity matrix,
    and a constant reaction coefficient $c = 1$
    on the square domain $\Omega_{\text{S}}$.
    Due to the small (isotropic) diffusion coefficient, we expect boundary layers to form along~$\partial\Omega_{\text{S}}$ in the true solution.
    As previously in Experiment~\ref{experiment-linear-1},
    in order to make sure that the convergence slope attains optimality
    already for comparatively small numbers of degrees of freedom,
    these boundary layers must be resolved by the energy-based adaptive refinement strategy.
    The energy norm of the true solution
    is approximated by the reference value
    \begin{equation}
        \|u\|_a^2
        \approx
        \underline{0.65094}51171871544.
    \end{equation}
        The convergence plots in Figure~\ref{fig:Exp2} indicate that the total error closely follows the optimal convergence behavior.
        For the energy tail-off
        stopping criterion~\eqref{eq:energy-tail-off-stopping-criterion}, we observe that
        the number of linear subspaces considered is visibly higher than in
        Experiment~\ref{experiment-linear-1}, which is, however, naturally accompanied by a
        smaller number of iterations per mesh.
        Moreover, as already observed in Experiment~\ref{experiment-linear-1},
        both stopping criteria display a single distinct outlier in the
        iteration count;
        this behavior appears to reflect the algorithm's tendency to further
        reduce the iteration error at certain points,
        thereby ensuring that it remains sufficiently below the total error.

    \begin{figure}
        \centering
        \includegraphics[width=0.7\textwidth]{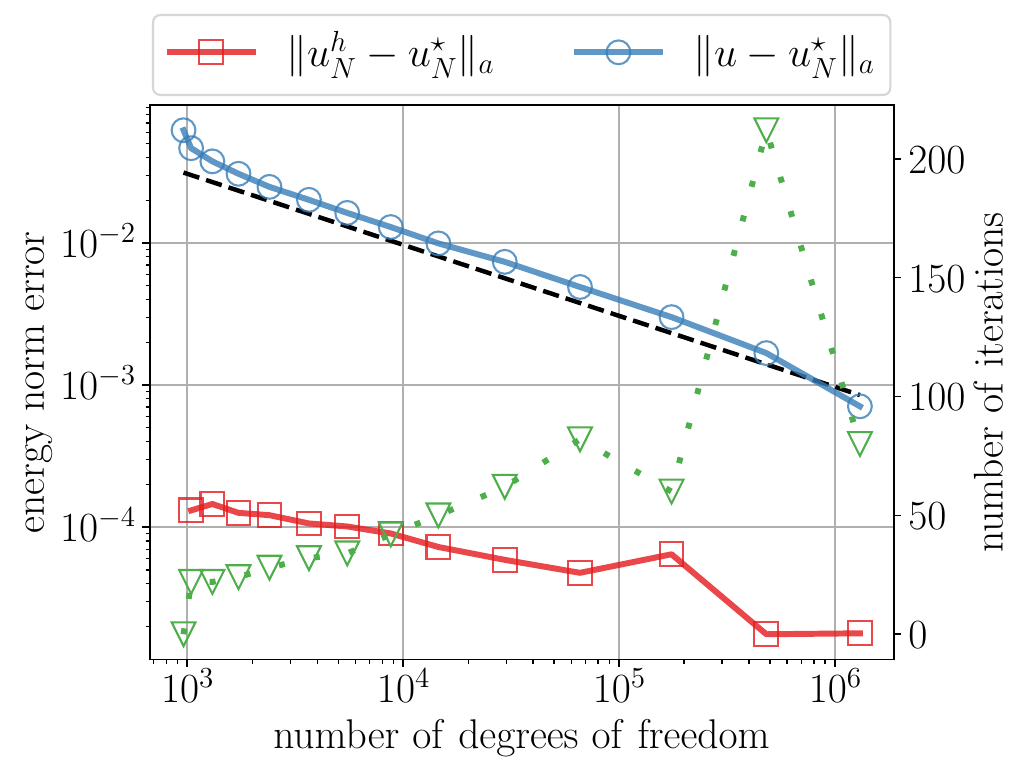} 
        \includegraphics[width=0.7\textwidth]{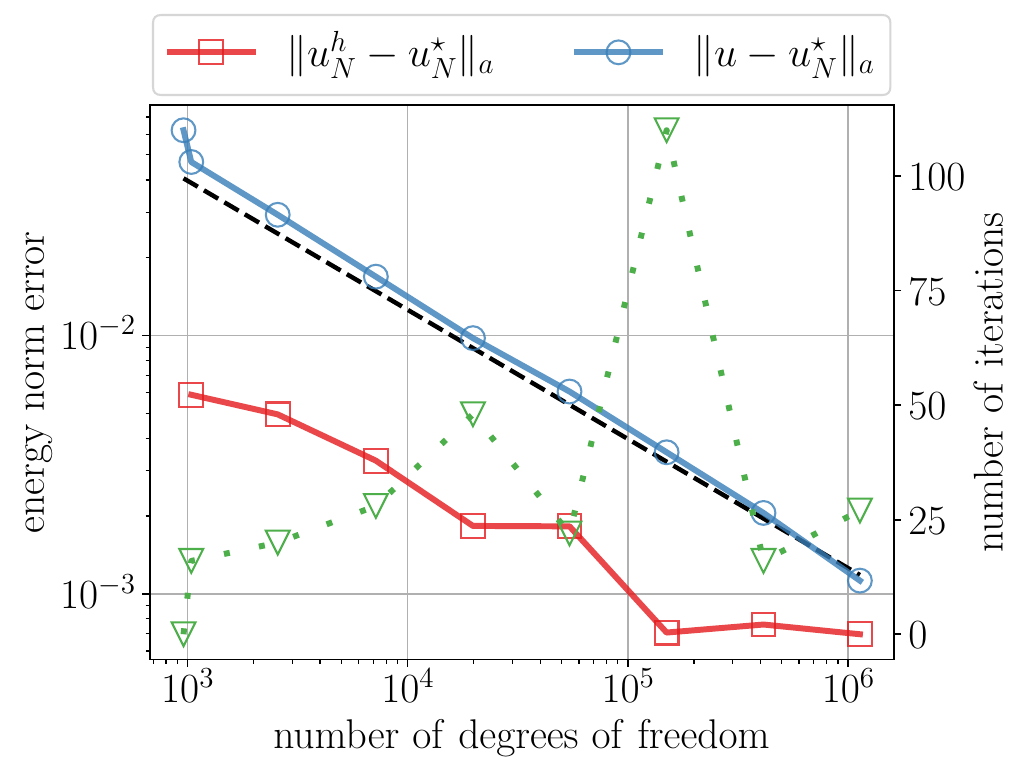} 
        \caption{
            Experiment~\ref{experiment-linear-2}.
            Top:
            energy tail-off
            stopping criterion~\eqref{eq:energy-tail-off-stopping-criterion} 
            with $n_{\text{min}} = 10$.
            Bottom:
            relative energy reduction
            stopping criterion~\eqref{eq:energy-arioli-stopping-criterion} 
            with
            $n_{\text{min}} = 5$,
            and $\alpha_\gamma = 0.1$.
        }
        \label{fig:Exp2}
    \end{figure}
\end{experiment}

\begin{experiment}
    \label{experiment-linear-3}
    Here, we consider the BVP~\eqref{eq:non-linear-pde}
    on the square domain $\Omega_{\text{S}}$, with a constant reaction coefficient $c=1$,
    and a piecewise constant diffusion matrix
    $\mat A(\mat x) = \kappa (\mat x) \mat 1_{2\times 2}$,
    where
    \begin{equation}
        \kappa(\mat x) =
        \begin{cases}
            1, & \mat x\in \Omega_{\text{S}} \setminus
            (\Omega_1 \cup \Omega_2 \cup \Omega_3), \\
            10, & \mat x \in \Omega_1, \\
            0.1, & \mat x \in \Omega_2, \\
            0.05, &  \mat x \in \Omega_3; 
        \end{cases}
    \end{equation}
    the subdomains $\Omega_1, \Omega_2, \Omega_3 \subset \Omega_{\text{S}}$
    are given by
    \begin{align}
        \Omega_1 = (0.1, 0.3) \times (0.1, 0.2), \qquad
        \Omega_2 = (0.4, 0.7) \times (0.1, 0.3), \qquad
        \Omega_3 = (0.8, 1.0) \times (0.7, 1.0).
    \end{align}
    Since the diffusion matrix has piecewise constant entries
    that differ by orders of magnitude when switching between subdomains,
    the interfaces of all subdomains must be adequately resolved
    by the adaptive refinement strategy;
    in our computations, the initial mesh
    has been aligned with the interfaces.
    In Figure~\ref{fig:Exp3-mesh} (left), we observe that our adaptive refinement
    strategy identifies the subdomain corners as
    particularly important regions to be resolved.
    Further, we stress that, due to the strongly varying diffusion 
    coefficient, the resulting system is ill-conditioned.
    Therefore, in this experiment only, we employed the diagonal 
    preconditioner $\mat P = \operatorname{diag}(\mat A)^{-1}$,
    where $\mat A$ denotes the left-hand side matrix
    in~\eqref{eq:linear-algebra-problem}.
    This was realized by passing $\mat P$ as an additional preconditioning
    parameter to \verb|scipy.sparse.linalg.cg|.
    As noted in Remark~\ref{remark:stopping-criteria},
    using any preconditioner does not interfere with the use
    of our custom stopping criteria.
    \begin{figure}
        \centering
        \includegraphics[width=0.49\textwidth]{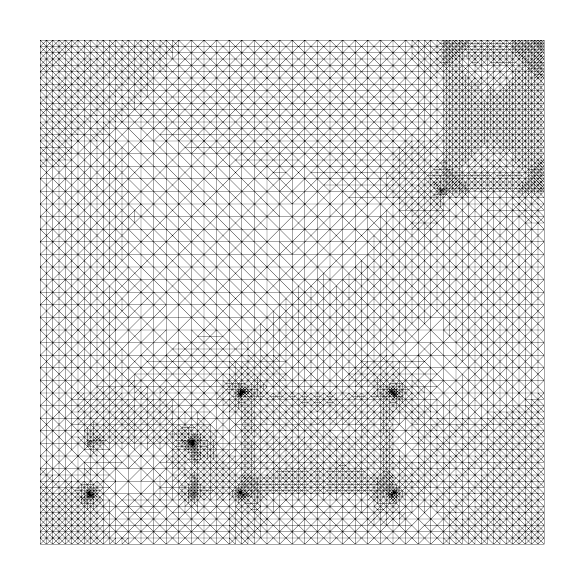}\hfill 
        \includegraphics[width=0.49\textwidth]{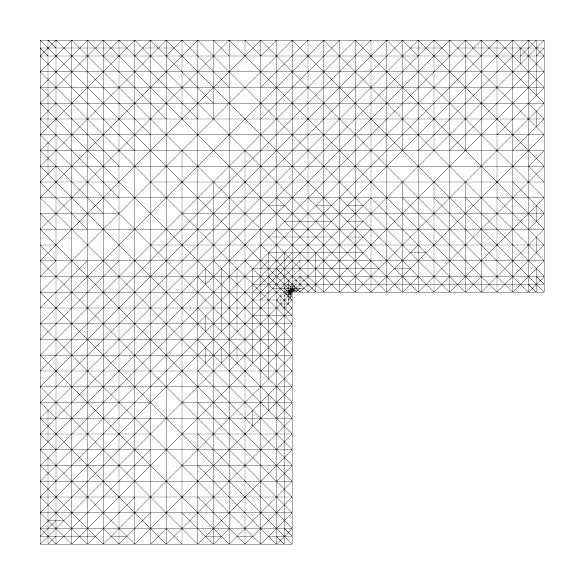}
        \caption{
            Both meshes were obtained
            in conjunction with the energy tail-off
            stopping criterion~\eqref{eq:energy-tail-off-stopping-criterion}
            and the control parameter $n_{\text{min}} = 10$.
            Left: Adaptively refined mesh for Experiment~\ref{experiment-linear-3}
            with $7511$ degrees of freedom.
            Right: Adaptively refined mesh for Experiment~\ref{experiment-linear-4}
            with $1330$ degrees of freedom.
        }
        \label{fig:Exp3-mesh}
    \end{figure}
    The energy norm of the true solution 
    is approximated by the reference value
    \begin{equation}
        \|u\|_a^2
        \approx
        \underline{0.040763}58619422494.
    \end{equation}
        The convergence plots show that the number of linear subspaces
        generated by the algorithm remains comparatively small for both
        stopping criteria.
        At the same time, the number of iterations per mesh stays low,
        except for the final data point in the case of stopping
        criterion~\eqref{eq:energy-tail-off-stopping-criterion},
        where a higher iteration count reduces the iteration error by
        roughly one order of magnitude below the total error.
        In both cases, our energy-based adaptive algorithm is able to attain optimal convergence.

    \begin{figure}
        \centering
        \includegraphics[width=0.7\textwidth]{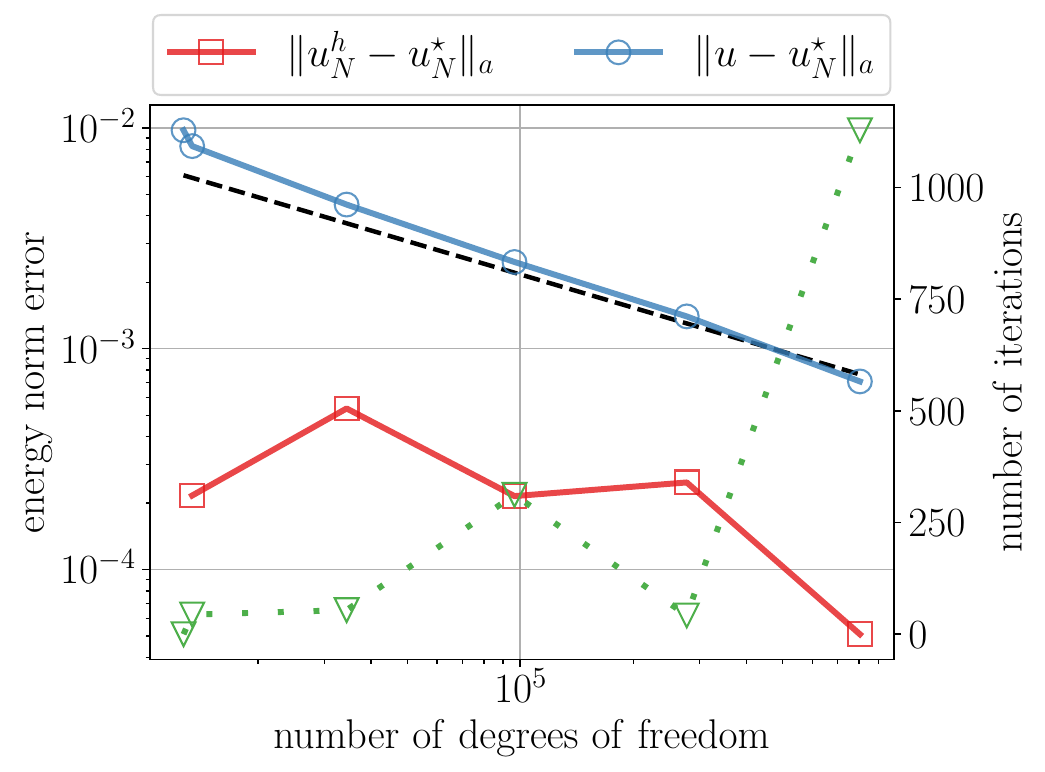} 
        \includegraphics[width=0.7\textwidth]{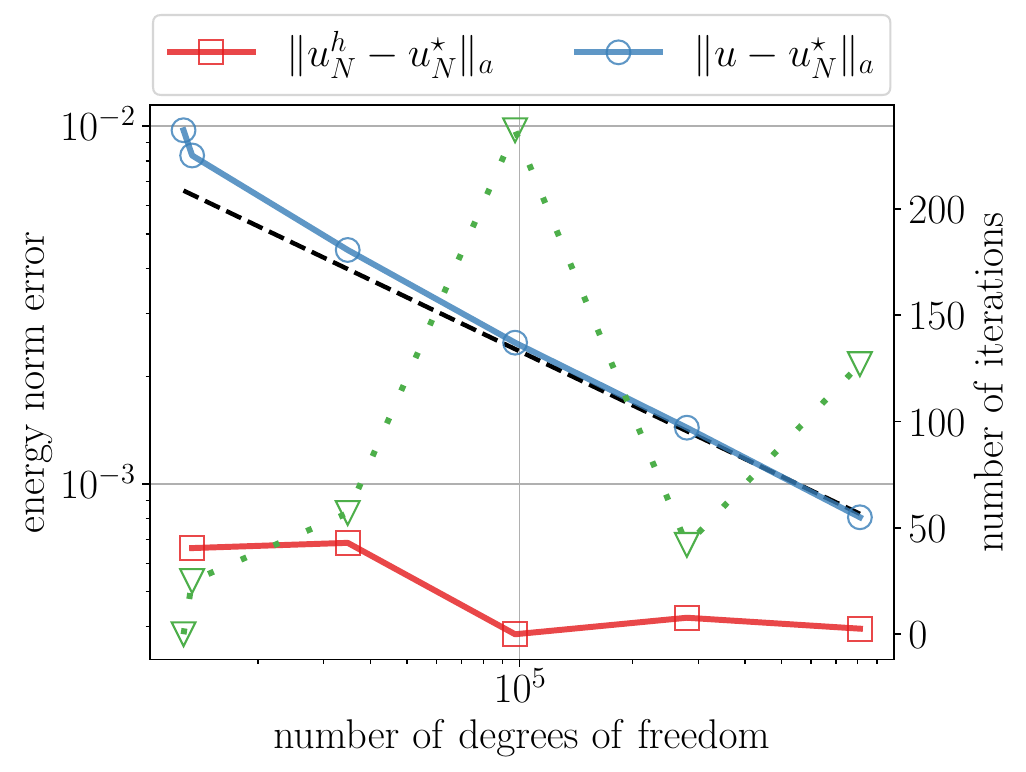} 
        \caption{
            Experiment~\ref{experiment-linear-3}.
            Top:
            energy tail-off
            stopping criterion~\eqref{eq:energy-tail-off-stopping-criterion} 
            with $n_{\text{min}} = 5$.
            Bottom:
            relative energy reduction
            stopping criterion~\eqref{eq:energy-arioli-stopping-criterion} 
            with
            $n_{\text{min}} = 5$,
            and $\alpha_\gamma = 0.1$.
        }
    \end{figure}
\end{experiment}

\begin{experiment}
    \label{experiment-linear-4}
    In this test, we consider the
    BVP~\eqref{eq:non-linear-pde} with a constant diffusion matrix
    $
        \mat A = \mat 1_{2\times 2},
    $
    and a vanishing reaction coefficient $c=0$
    on the L-shaped domain $\Omega_{\text{L}}$.
    This benchmark problem is well-known to exhibit an elliptic corner
    singularity at the origin that mandates appropriate local mesh refinement,
    which is correctly resolved by our adaptive mesh refinement strategy,
    as can be observed in Figure~\ref{fig:Exp3-mesh} (right).
    The energy norm of the true solution 
    is approximated by 
    \begin{equation}
        \|u\|_a^2
        \approx
        \underline{0.21407580}2220546;
    \end{equation}
    this value was obtained\footnote{The authors gratefully acknowledge the help of Dr.~Patrick Bammer for carrying
    out the necessary high-accurate $hp$-FEM computations.} by applying a standard
    $hp$-finite element approach as described in \cite[\S4]{schwab98}.
    Our energy-based adaptive refinement strategy successfully identifies the singular behavior in the vicinity of the origin, see Figure~\ref{fig:Exp4},
    thereby ensuring optimal convergence for both stopping criteria of the CG iteration; incidentally, in either case
    we observe two distinct peaks in the iteration count,
    which appear to enforce sufficient reduction of the iteration error.
    As a result, the iteration error remains roughly two orders of
    magnitude below the total error for the energy tail-off
    stopping criterion~\eqref{eq:energy-tail-off-stopping-criterion},
    and about one order of magnitude below for the relative energy reduction
    stopping criterion~\eqref{eq:energy-arioli-stopping-criterion}.

    \begin{figure}
        \centering
        \includegraphics[width=0.7\textwidth]{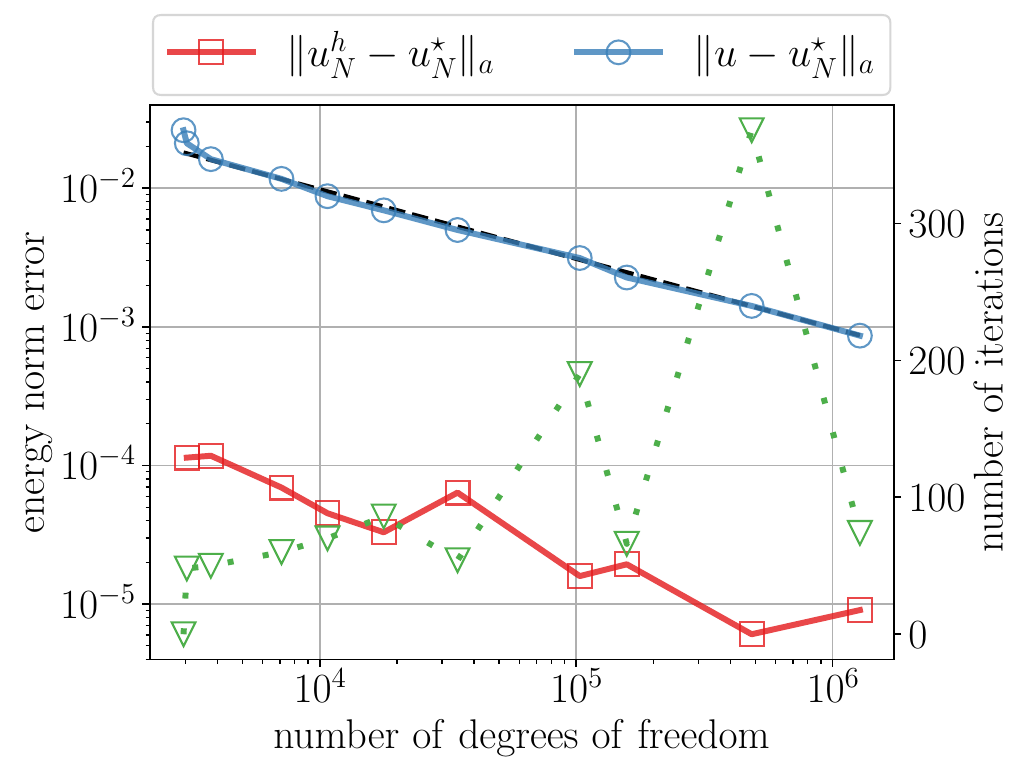}
        \includegraphics[width=0.7\textwidth]{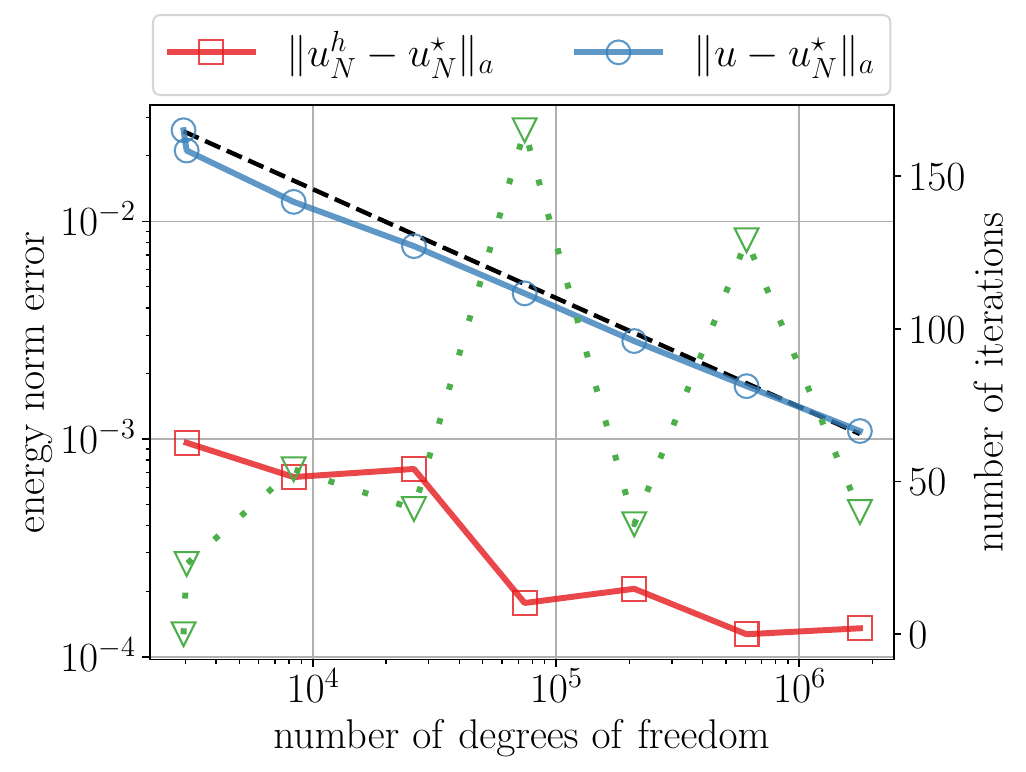}
        \caption{
            Experiment~\ref{experiment-linear-4}.
            Top:
            energy tail-off
            stopping criterion~\eqref{eq:energy-tail-off-stopping-criterion} 
            with $n_{\text{min}} = 10$.
            Bottom:
            relative energy reduction
            stopping criterion~\eqref{eq:energy-arioli-stopping-criterion} 
            with
            $n_{\text{min}} = 10$,
            and $\alpha_\gamma = 0.01$.
        }
        \label{fig:Exp4}
    \end{figure}

\end{experiment}

\section{Conclusions}
\label{sec:conclusions}

We have developed a fully iterative, energy-based adaptive
framework for the numerical minimization of strictly convex and weakly coercive energy
functionals $\E : \spc V \to \mathbb{R}$ on a Hilbert space~$\spc V$.
Roughly speaking, the proposed strategy combines
an energy-based adaptive construction of a
hierarchical sequence of finite-dimensional subspaces
with an iterative energy minimization strategy on each discrete
subspace $\spc V_N\subset\spc V, N\in\mathbb{N}$,
based on a (nonlinear) CG method in conjunction
with energy-based stopping criteria.

Within a more specific framework, we have presented a straightforward
realization of our abstract strategy for $\spc P_1$ finite element discretizations of
the underlying Hilbert space $\spc V = \Hone$,
termed edge-based variational adaptivity.
In this context,
compared to earlier works, our contribution provides essentially two novelties:
Firstly, with respect to adaptive enrichment,
previous implementations of \emph{variational adaptivity}
compute local energy reductions using element-based refinements,
see, e.g.,~\cite{ahw23,HSW21}, whilst in the current work, we exploit the corresponding energy reductions via edge bisection, which are slightly more localized.
Secondly, regarding energy minimization on each finite-dimensional subspace,
our approach intertwines the adaptive space enrichments with a (nonlinear) CG method,
again purely based on energy techniques, to solve the discrete minimization problems, and thereby
inherently avoids the need to solve any large-scale
linear system exactly.
%
Our numerical experiments confirm that,
even in this minimal setting of
\emph{edge-based variational adaptivity}
and without solving any large-scale linear system exactly,
the resulting sequence of final iterates
$\{u_N^\star\}_N \subset \Hone$
exhibits optimal convergence rates.

\clearpage
\printbibliography

\end{document}